\documentclass{birkmult}
\usepackage{amsthm,amsmath,amsfonts,latexsym,amssymb,mathrsfs}

\theoremstyle{definition}
\newtheorem{Def}{Definition}[section]
\newtheorem{df}[Def]{Definition}
\newtheorem{thm}[Def]{Theorem}
\newtheorem{cor}[Def]{Corollary}
\newtheorem{lem}[Def]{Lemma}

\newtheorem{rem}[Def]{Remark}
\newtheorem{ex}[Def]{Example}

\renewcommand{\thefootnote}{\fnsymbol{footnote}}

\begin{document}

\title[Abrahamse's Theorem and Subnormal Toeplitz Completion]
{\bf Abrahamse's Theorem for matrix-valued symbols\\
and subnormal Toeplitz completions}
\author{Ra{\'u}l\ E.\ Curto}
\address{Department of Mathematics, University of Iowa, Iowa City, IA 52242, U.S.A.}
\email{raul-curto@uiowa.edu}

\author{In Sung Hwang}
\address{Department of Mathematics, Sungkyunkwan University, Suwon 440-746, Korea}
\email{ihwang@skku.edu}

\author{Woo Young Lee}
\address{Department of Mathematics, Seoul National University, Seoul 151-742, Korea}
\email{wylee@snu.ac.kr}

\date{}
\maketitle

\renewcommand{\thefootnote}{}
\footnote{
The work of the first named author was partially supported by NSF
Grant DMS-0801168. \ The work of second named author was supported
by Basic Science Research Program through NRF funded by the Ministry
of Education, Science and Technology (2011-0022577). \ The work of
the third author was supported by the National Research Foundation
of Korea(NRF) grant funded by the Korea government(MEST)
(No.2012-0000939).}

\maketitle

\bigskip

\begin{abstract}
This paper deals with subnormality of Toeplitz operators with
matrix-valued symbols and, in particular, with
an appropriate reformulation of Halmos's Problem 5: Which subnormal Toeplitz operators with
matrix-valued symbols are either normal or analytic\,? \
In 1976, M. Abrahamse showed that if $\varphi\in L^\infty$ is such that $\varphi$
or $\overline\varphi$ is of bounded type and if $T_\varphi$ is subnormal, then
$T_\varphi$ is either normal or analytic. \
In this paper we establish a matrix-valued version of Abrahamse's
Theorem and then apply this result to solve the following Toeplitz
completion problem: Find the unspecified Toeplitz entries of the
partial block Toeplitz matrix
$$
A:=\begin{bmatrix} T_{\overline b_\alpha} & ?\\ ?& T_{\overline
b_\beta}\end{bmatrix}\quad\hbox{($\alpha,\beta\in\mathbb D$)}
$$
so that $A$ becomes subnormal, where $b_\lambda$ is a Blaschke
factor of the form $b_\lambda(z):=\frac{z-\lambda}{1-\overline
\lambda z}$ ($\lambda\in \mathbb D$).
\end{abstract}

%
%
\pagestyle{plain}

\bigskip

\section{Introduction}

\noindent This paper focuses on subnormality for
Toeplitz operators with matrix-valued symbols and more precisely,
the case of Toeplitz
operators with matrix-valued bounded type symbols. \
In this paper we give an
appropriate generalization of Abrahamse's Theorem to the case of
matrix-valued symbols and apply this generalization to solve a
subnormal Toeplitz completion problem. \

\smallskip

To describe our results in more detail, we first need to review a
few essential facts about (block) Toeplitz operators, and for that
we will use \cite{BS}, \cite{Do1}, \cite{Do2}, \cite{GGK}, \cite{Ni} and \cite{Pe}. \
Let
$\mathcal{H}$ be a complex Hilbert space and let $\mathcal{B(H)}$ be
the algebra of bounded linear operators acting on $\mathcal{H}$. \
An operator $T\in\mathcal{B(H)}$ is said to be {\it hyponormal} if
its self-commutator $[T^*,T]:= T^*T-TT^*$ is positive
(semi-definite), and {\it subnormal} if there exists a normal
operator $N$ on some Hilbert space $\mathcal{K}\supseteq
\mathcal{H}$ such that $\mathcal H$ is invariant under $N$ and
$N\vert_{\mathcal{H}}=T$. \
Let $\mathbb{T} \equiv
\partial\,\mathbb{D}$ be the unit circle in the complex plane. \
Let
$L^2\equiv L^2({\mathbb T})$ be the set of all square-integrable
measurable functions on $\mathbb{T}$ and let $H^2\equiv H^2({\mathbb
T})$ be the corresponding Hardy space. \
Let $H^\infty\equiv
H^\infty(\mathbb T):=L^\infty (\mathbb T)\cap H^2 (\mathbb T)$, that
is, $H^\infty$ is the set of bounded analytic functions on $\mathbb D$. \
Given $\varphi\in L^\infty$, the Toeplitz operator $T_\varphi$ and
the Hankel operator $H_\varphi$ are defined by
$$
T_\varphi g:=P(\varphi g) \quad\hbox{and}\quad H_\varphi g:=JP^\perp(\varphi g)
\qquad (g\in H^2),
$$
where $P$ and $P^\perp$ denote the orthogonal projections that map
from $L^2$ onto $H^2$ and $(H^2)^\perp$, respectively, and where $J$
denotes the unitary operator on $L^2$ defined by $J(f)(z)=\overline
z f(\overline z)$. \
In 1988, the hyponormality of $T_\varphi$ was completely characterized
in terms of its symbol via Cowen's Theorem \cite{Co3}.

\bigskip

\noindent {\bf Cowen's Theorem.} (\cite{Co3}, \cite{NT}) {\it For
each $\varphi\in L^\infty$, let
$$
\mathcal{E}(\varphi)\equiv \{k\in H^\infty:\ ||k||_\infty\le 1\
\hbox{and}\ \varphi-k\overline\varphi\in H^\infty\}.
$$
Then $T_\varphi$ is hyponormal if and only if $\mathcal{E}(\varphi)$ is
nonempty. }
\bigskip

This elegant and useful theorem has been used in
\cite{CuL1}, \cite{CuL2}, \cite{FL}, \cite{Gu1}, \cite{Gu2},
\cite{GS}, \cite{HKL1}, \cite{HKL2}, \cite{HL1}, \cite{HL2},
\cite{HL3}, \cite{Le}, \cite{NT} and \cite{Zhu}, which have been
devoted to the study of hyponormality for Toeplitz operators on $H^2$. \
When one studies the hyponormality (also, normality and
subnormality) of the Toeplitz operator $T_\varphi$ one may, without
loss of generality, assume that $\varphi(0)=0$; this is because
hyponormality is invariant under translation by scalars. \

\medskip

We now recall that a function $\varphi\in L^\infty$ is said to be of {\it bounded type} (or
in the Nevanlinna class) if there are analytic functions
$\psi_1,\psi_2\in H^\infty (\mathbb D)$ such that
$$
\varphi(z)=\frac{\psi_1(z)}{\psi_2(z)}\quad\hbox{for almost all}\ z\in
\mathbb{T}.
$$
If $\varphi\in L^\infty$, we write
$$
\varphi_+\equiv P \varphi\in H^2\quad\hbox{and}\quad \varphi_-\equiv
\overline{P^\perp \varphi}\in zH^2.
$$
Let $BMO$ denote the set of functions of bounded mean oscillation in
$L^1$. Then $L^\infty\subseteq BMO \subseteq L^2$. \
It is well-known
that if $f\in L^2$, then $H_f$ is bounded on $H^2$ whenever $P^\perp
f\in BMO$ (cf. \cite{Pe}). \
If $\varphi\in L^\infty$, then
$\overline{\varphi_-}, \overline{\varphi_+}\in BMO$, so that
$H_{\overline{\varphi_-}}$ and $H_{\overline{\varphi_+}}$ are well understood. \
It is well known \cite[Lemma 3]{Ab} that if $\varphi\in L^\infty$ then
\begin{equation}\label{1.1}
\hbox{$\varphi$ is of bounded type}\ \Longleftrightarrow\ \hbox{ker}\,
H_\varphi\ne \{0\}\,.
\end{equation}
Assume now that both $\varphi$ and $\overline\varphi$ are of bounded type. \
Since $T_{\overline z}H_\psi=H_\psi T_z$ for all $\psi \in
L^{\infty}$, it follows from Beurling's Theorem that $\text{ker}\,
H_{\overline{\varphi_-}}=\theta_0 H^2$ and $\text{ker}\,
H_{\overline{\varphi_+}}=\theta_+ H^2$ for some inner functions $\theta_0, \theta_+$. \
We thus have $b:={\overline{\varphi_-}}\theta_0
\in H^2$, and hence we can write
\begin{equation}\label{1.1-1}
\varphi_-=\theta_0\overline{b}, \text{~and similarly~}
\varphi_+=\theta_+\overline{a} \text{~for some~} a \in H^2.
\end{equation}
In the factorization (\ref{1.1-1}), we will always assume that
$\theta_0$ and $b$ are coprime and $\theta_+$ and $a$ are coprime. \
In particular, if $T_\varphi$ is hyponormal and $\varphi\notin H^\infty$, and since
$$
[T_\varphi^*, T_\varphi]=H_{\overline\varphi}^* H_{\overline\varphi}-H_\varphi^*
H_\varphi= H_{\overline{\varphi_+}}^*
H_{\overline{\varphi_+}}-H_{\overline{\varphi_-}}^* H_{\overline{\varphi_-}},
$$
it follows that $||H_{\overline{\varphi_+}} f||\ge
||H_{\overline{\varphi_-}} f||$ for all $f\in H^2$, and hence
\begin{equation*}
\theta_+ H^2= \text{ker}\, H_{\overline{\varphi_+}}\subseteq
\text{ker}\, H_{\overline{\varphi_-}}=\theta_0 H^2,
\end{equation*}
which implies that $\theta_0$ divides $\theta_+$, i.e.,
$\theta_+=\theta_0\theta_1$ for some inner function $\theta_1$. \
We write, for an inner function $\theta$,
$$
\mathcal H_{\theta}:=H^2\ominus \theta\,H^2.
$$
Note that if $f=\theta \overline a \in L^2$, then $f\in H^2$ if and
only if $a\in \mathcal H_{z\theta}$; in particular, if $f(0)=0$ then
$a\in \mathcal H_\theta$. \
Thus, if $\varphi=\overline{\varphi_-}+\varphi_+\in
L^\infty$ is such that $\varphi$ and $\overline\varphi$ are of bounded
type and $T_\varphi$ is hyponormal, then we can write
\begin{equation}\label{1.3}
\varphi_+=\theta_0\theta_1\overline a\quad\text{and}\quad
\varphi_-=\theta_0 \overline b, \qquad\text{where $a\in
\mathcal{H}_{z\theta_0\theta_1}$ and $b\in\mathcal{H}_{\theta_0}$:}
\end{equation}
in this case, $\theta_0\theta_1\overline a$ and $\theta_0 \overline
b$ are called {\it coprime factorizations} of $\varphi_+$ and $\varphi_-$,
respectively. \
By Kronecker's Lemma \cite[p. 183]{Ni}, if $f\in
H^\infty$ then $\overline f$ is a rational function if and only if
$\hbox{rank}\, H_{\overline f}<\infty$, which implies that
\begin{equation}\label{1.4}
\hbox{$\overline{f}$ is rational} \ \Longleftrightarrow\
f=\theta\overline b\ \ \hbox{with a finite Blaschke product
$\theta$}.
\end{equation}

\bigskip

We now introduce the notion of block Toeplitz operators. \
For a
Hilbert space $\mathcal X$, let $L^2_{\mathcal X}\equiv
L^2_{\mathcal X}(\mathbb T)$ be the Hilbert space of $\mathcal
X$-valued norm square-integrable measurable functions on
$\mathbb{T}$ and let $H^2_{\mathcal X}\equiv H^2_{\mathcal
X}(\mathbb T)$ be the corresponding Hardy space. \
We observe that
$L^2_{\mathbb{C}^n}= L^2\otimes \mathbb{C}^n$ and
$H^2_{\mathbb{C}^n}= H^2\otimes \mathbb{C}^n$. If $\Phi$ is a
matrix-valued function in $L^\infty_{M_n}\equiv
L^\infty_{M_n}(\mathbb T)$ ($=L^\infty\otimes M_n$) then $T_\Phi:
H^2_{\mathbb{C}^n}\to H^2_{\mathbb{C}^n}$ denotes the block Toeplitz
operator with symbol $\Phi$ defined by
$$
T_\Phi F:=P_n(\Phi F)\quad \hbox{for}\ F\in H^2_{\mathbb{C}^n},
$$
where $P_n$ is the orthogonal projection of $L^2_{\mathbb{C}^n}$
onto $H^2_{\mathbb{C}^n}$. \
A block Hankel operator with symbol
$\Phi\in L^\infty_{M_n}$ is the operator $H_\Phi:
H^2_{\mathbb{C}^n}\to H^2_{\mathbb{C}^n}$ defined by
$$
H_\Phi F := J_n P_n^\perp (\Phi F)\quad \hbox{for}\ F\in
H^2_{\mathbb{C}^n},
$$
where $P_n^\perp$ is the orthogonal projection of $L^2_{\mathbb{C}^n}$
onto $(H^2_{\mathbb{C}^n})^\perp$, $J_n$ denotes the unitary operator on
$L^2_{\mathbb{C}^n}$ given
by $J_n(F)(z):=\overline{z} I_n F(\overline{z})$ for $F \in
L^2_{\mathbb{C}^n}$, and where $I_n$ is the $n\times n$ identity matrix. \
For $\Phi\in L^\infty_{M_n}$, we write
$$
\widetilde\Phi (z):=\Phi^*(\overline z).
$$
For $\Phi\in L^\infty_{M_n}$, we also write
$$
\Phi_+:=P_n \Phi \in H^2_{M_n} \quad\hbox{and}\quad
\Phi_-:=\bigl(P_n^\perp \Phi\bigr)^* \in H^2_{M_n}.
$$
Thus we can write $\Phi=\Phi_-^*+\Phi_+\,$. \
However, it will be often convenient to permit the constant
term for $\Phi_-$. \
Hence, if there is no confusion we may
assume that $\Phi_-$ shares the constant term with $\Phi_+$: in this case,
$\Phi(0)=\Phi_+(0)+\Phi_-(0)^*$. \

A matrix-valued function
$\Theta\in H^\infty_{M_{n\times m}}$ ($=H^\infty\otimes M_{n\times
m}$) is called {\it inner} if $\Theta(z)^*\Theta(z)=I_m$ for almost
all $z\in\mathbb{T}$. \
The following basic relations can be easily derived:
\begin{align}
&T_\Phi^*=T_{\Phi^*},\ \  H_\Phi^*= H_{\widetilde
\Phi}\quad (\Phi\in L^\infty_{M_n});\notag\\
&T_{\Phi\Psi}-T_\Phi T_\Psi = H_{\Phi^*}^*H_\Psi \quad (\Phi,\Psi\in L^\infty_{M_n});\label{1.5}\\
&H_\Phi T_\Psi = H_{\Phi\Psi},\ \
H_{\Psi\Phi}=T_{\widetilde{\Psi}}^*H_\Phi\quad (\Phi\in
L^\infty_{M_n}, \Psi\in H^\infty_{M_n});\label{1.5-1}
\end{align}

\noindent
For a matrix-valued function $\Phi=[\phi_{ij}]\in
L^\infty_{M_n}$, we say that $\Phi$ is of {\it bounded type} if each
entry $\phi_{ij}$ is of bounded type and that $\Phi$ is {\it
rational} if each entry $\phi_{ij}$ is a rational function. \

\medskip

In 2006, Gu, Hendricks and Rutherford \cite{GHR} characterized the
hyponormality of block Toeplitz operators in terms of their symbols. \
In particular they showed that if $T_\Phi$ is a hyponormal
block Toeplitz operator on $H^2_{\mathbb{C}^n}$, then $\Phi$ is
normal, i.e., $\Phi^*\Phi=\Phi\Phi^*$. \
Their characterization for
hyponormality of block Toeplitz operators resembles Cowen's
Theorem except for an additional condition -- the normality
condition of the symbol.

\medskip

\noindent
\begin{lem}\label{lem1.2}({\bf Hyponormality of Block Toeplitz Operators})
(Gu-Hendricks-Rutherford \cite{GHR}) For each $\Phi\in
L^\infty_{M_n}$, let
$$
\mathcal{E}(\Phi):=\Bigl\{K\in H^\infty_{M_n}:\ ||K||_\infty \le 1\
\ \hbox{and}\ \ \Phi-K \Phi^*\in H^\infty_{M_n}\Bigr\}.
$$
Then $T_\Phi$ is hyponormal if and only if $\Phi$ is normal and
$\mathcal{E}(\Phi)$ is nonempty.
\end{lem}

\bigskip

In \cite{GHR}, the normality of block Toeplitz operator $T_\Phi$ was
also characterized in terms of the symbol $\Phi$, under a
``determinant" assumption on the symbol $\Phi$. \

\medskip

\noindent
\begin{lem}\label{lem1.3}({\bf Normality of Block Toeplitz Operators})
(Gu-Hendricks-Rutherford \cite{GHR}) Let $\Phi\equiv \Phi_+
+\Phi_-^*$ be normal. \
If
$\hbox{det}\,\Phi_+$ is not identically zero then
\begin{equation}\label{1.3-1}
T_\Phi\ \hbox{is normal}\Longleftrightarrow \Phi_+-\Phi_+(0)=\bigl(\Phi_- - \Phi_-(0)\bigr)\,U\ \
\hbox{for some constant unitary matrix}\ U.
\end{equation}
\end{lem}

\bigskip

On the other hand, M. Abrahamse \cite[Lemma 6]{Ab} showed that if
$T_\varphi$ is hyponormal, if $\varphi\notin H^\infty$, and if
$\varphi$ or $\overline{\varphi}$ is of bounded type then both
$\varphi$ and $\overline{\varphi}$ are of bounded type. \
However,
by contrast to the scalar case, $\Phi^*$ may not be of bounded type
even though $T_\Phi$ is hyponormal, $\Phi\notin H^\infty_{M_n}$ and
$\Phi$ is of bounded type. \
But we have a one-way implication
(see \cite [Corollary 3.5 and Remark 3.6]{GHR}):
\begin{equation}\label{1.3-3}
\hbox{$T_\Phi$ is hyponormal and $\Phi^*$ is of bounded type}\ \Longrightarrow\
\hbox{$\Phi$ is of bounded type.}
\end{equation}

\medskip

For a matrix-valued function $\Phi\in H^2_{M_{n\times r}}$, we say
that $\Delta\in H^2_{M_{n\times m}}$ is a {\it left inner divisor}
of $\Phi$ if $\Delta$ is an inner matrix function such that
$\Phi=\Delta A$ for some $A \in H^{2}_{M_{m\times r}}$ ($m\le n$). \
We also say that two matrix functions $\Phi\in H^2_{M_{n\times r}}$
and $\Psi\in H^2_{M_{n\times m}}$ are {\it left coprime} if the only
common left inner divisor of both $\Phi$ and $\Psi$ is a unitary
constant and that $\Phi\in H^2_{M_{n\times r}}$ and $\Psi\in
H^2_{M_{m\times r}}$ are {\it right coprime} if $\widetilde\Phi$ and
 $\widetilde\Psi$ are left coprime. \
Two matrix functions $\Phi$ and $\Psi$ in $H^2_{M_n}$ are said to be
{\it coprime} if they are both left and right coprime. \
We remark
that if $\Phi\in H^2_{M_n}$ is such that $\hbox{det}\,\Phi$ is not
identically zero then any left inner divisor $\Delta$ of $\Phi$ is
square, i.e., $\Delta\in H^2_{M_n}$. \
If $\Phi\in H^2_{M_n}$ is such
that $\hbox{det}\,\Phi$ is not identically zero then we say that
$\Delta\in H^2_{M_{n}}$ is a {\it right inner divisor} of $\Phi$ if
$\widetilde{\Delta}$ is a left inner divisor of $\widetilde{\Phi}$.

\medskip

The following lemma will be useful in the sequel.

\noindent
\begin{lem}\label{lem1.1}(\cite{GHR})\
For $\Phi\in L^\infty_{M_n}$, the following statements are
equivalent:
\medskip

{\rm (i)} $\Phi$ is of bounded type;

{\rm (ii)} $\hbox{ker}\, H_\Phi=\Theta H^2_{\mathbb{C}^n}$ for some
square inner matrix function $\Theta$;

{\rm (iii)} $\Phi=A\Theta^*$, where $A\in H^{\infty}_{M_n}$ and $A$
and $\Theta$ are right coprime.
\end{lem}

\medskip

\noindent For an inner matrix function $\Theta\in H^\infty_{M_n}$,
write
$$
\mathcal{H}_{\Theta}:=\left(\Theta H^2_{\mathbb{C}^n}\right)^\perp
\equiv H^2_{\mathbb{C}^n}\ominus \Theta H^2_{\mathbb{C}^n}.
$$
Suppose $\Phi=[\varphi_{ij}] \in L^\infty_{M_n}$ is such that $\Phi^*$
is of bounded type. \
Then we may write
$\varphi_{ij}=\theta_{ij}\overline{b}_{ij}$, where $\theta_{ij}$ is an
inner function and $\theta_{ij}$ and $b_{ij}$ are coprime. \
Thus if
$\theta$ is the least common multiple of $\theta_{ij}$'s (i.e., the
$\theta_{ij}$ divide $\theta$ and if they divide an inner function
$\theta^\prime$ then $\theta$ in turn divides $\theta^\prime$), then
we can write
\begin{equation}\label{1.9}
\Phi=[\varphi_{ij}]=[\theta_{ij}\overline{b}_{ij}]=[\theta
\overline{a}_{ij}]= \Theta A^* \quad (\Theta=\theta I_n,\ A \in
H^{2}_{M_n}).
\end{equation}
We note that the representation (\ref{1.9}) is ``minimal," in the
sense that if $\omega I_n$ ($\omega$ is inner) is a common inner
divisor of $\Theta$ and $A$, then $\omega$ is constant. \
Let
$\Phi\equiv \Phi_-^*+\Phi_+\in L^\infty_{M_n}$ be such that $\Phi$
and $\Phi^*$ are of bounded type. \
Then in view of (\ref{1.9}) we can write
$$
\Phi_+= \Theta_1 A^* \quad\hbox{and}\quad \Phi_-= \Theta_2 B^*,
$$
where $\Theta_i =\theta_i I_n$ with an inner function $\theta_i$ for
$i=1,2$ and $A,B\in H^{2}_{M_n}$. \
In particular, if $\Phi\in
L^\infty_{M_n}$ is rational then the $\theta_i$ are chosen as finite
Blaschke products as we observed in (\ref{1.4}).
\bigskip

In this paper we consider the subnormality of block Toeplitz
operators and in particular, the matrix-valued version of Halmos's
Problem 5: Which subnormal Toeplitz operators with matrix-valued
symbols are either normal or analytic\,? \
In 1976, M. Abrahamse
showed that if $\varphi\in L^\infty$ is
such that $\varphi$ or $\overline \varphi$ is of bounded type, if $T_\varphi$
is hyponormal, and if $\hbox{\rm ker}\, [T_\varphi^*, T_\varphi]$ is
invariant under $T_\varphi$ then $T_\varphi$ is either normal or analytic. \
The
purpose of this paper is to establish a matrix-valued version of
Abrahamse's Theorem and then apply this result to solve a Toeplitz
completion problem. \
In Section 2 we make a brief sketch on
Halmos's Problem 5 and the earlier results. \
Section 3 is devoted to
get an Abrahamse's Theorem for matrix-valued symbols. \
In Section 4,
using our extension of Abrahamse's Theorem for matrix-valued
symbols, we solve the following `Toeplitz completion"
problem: find the unspecified Toeplitz entries of the partial block
Toeplitz matrix
$$
A:=\begin{bmatrix} T_{\overline b_\alpha} & ?\\ ?& T_{\overline
b_\beta}\end{bmatrix}\quad\hbox{($\alpha,\beta\in\mathbb D$)}
$$
so that $A$ becomes subnormal, where $b_\lambda$ is a Blaschke
factor of the form $b_\lambda(z):=\frac{z-\lambda}{1-\overline
\lambda z}$ ($\lambda\in \mathbb D$).

\vskip 1cm

%
%

\section{Halmos's Problem 5 and Abrahamse's Theorem}

\noindent In 1970, P.R. Halmos posed the following problem, listed
as Problem 5 in his lecture ``Ten problems in Hilbert space"
\cite{Hal1}, \cite{Hal2}:
$$
\hbox{Is every subnormal Toeplitz operator either normal or
analytic\,?}
$$
A Toeplitz operator $T_\varphi$ is called {\it analytic} if $\varphi\in H^\infty$. \
Any analytic Toeplitz operator is easily seen to be
subnormal: indeed, $T_\varphi h=P(\varphi h)=\varphi h =M_\varphi h$ for $h\in
H^2$, where $M_\varphi$ is the normal operator of multiplication by
$\varphi$ on $L^2$. \
The question is natural because the two classes,
the normal and analytic Toeplitz operators, are fairly well
understood and are subnormal. \
In 1984, Halmos's Problem 5 was
answered in the negative by C. Cowen and J. Long \cite{CoL}. \
However, unfortunately, Cowen and Long's construction does not
provide an intrinsic connection between subnormality and the theory
of Toeplitz operators. \
Until now researchers have been unable to
characterize subnormal Toeplitz operators in terms of their
symbols. \

We would like to reformulate Halmos's Problem 5 as follows:
\medskip

\noindent {\bf Halmos's Problem 5 reformulated.} {\it Which Toeplitz
operators are subnormal\,?}

\bigskip

The most interesting partial answer to Halmos's Problem 5 was given
by M. Abrahamse \cite{Ab}. \
M. Abrahamse gave a general sufficient
condition for the answer to Halmos's Problem 5 to be affirmative.
Abrahamse's Theorem can be then stated as:

\medskip

\noindent{\bf Abrahamse's Theorem} (\cite[Theorem]{Ab}). {\it Let
$\varphi\in L^\infty$ be such that $\varphi$
or $\overline \varphi$ is of bounded type. \
If $T_\varphi$ is hyponormal
and $\hbox{\rm ker}\,[T_\varphi^*, T_\varphi]$ is invariant under $T_\varphi$
then $T_\varphi$ is normal or analytic. }
\bigskip

Consequently, if $\varphi\in L^\infty$ is
such that $\varphi$ or $\overline \varphi$ is of bounded type, then every
subnormal Toeplitz operator must be either normal or analytic. \

We say that a block Toeplitz operator $T_\Phi$ is {\it analytic} if
$\Phi\in H^\infty_{M_n}$. \
Evidently, any analytic block Toeplitz
operator with a normal symbol is subnormal because the
multiplication operator $M_\Phi$ is a normal extension of $T_\Phi$. \
As a first inquiry in the above reformulation of Halmos's Problem 5
the following question can be raised:
$$
\hbox{Is Abrahamse's Theorem valid for block Toeplitz operators}\,?
$$
In [CHL2, Theorem 3.5], the authors gave a matrix-valued version of
Abrahamse's Theorem. As a corollary the following result was shown:

\medskip

\begin{thm}\label{thm2.1} ([CHL2, Corollary 3.9]).
Suppose $\Phi=\Phi_-^*+ \Phi_+\in L^\infty_{M_n}$ is a matrix-valued
rational function. \
Then in view of (\ref{1.9}) and (\ref{1.4}), we may write
\begin{equation}\label{2.1}
\Phi_- = B^*\Theta,
\end{equation}
where $\Theta:=\theta I_n$ with a finite Blaschke product $\theta$. \
Assume that $B$ and $\Theta$ are (left) coprime. If $T_{\Phi}$ is
subnormal then $T_{\Phi}$ is either normal or analytic.
\end{thm}

\bigskip

Note that in the coprime factorization (\ref{2.1}) of $\Phi_-$,
$\Theta$ is a {\it diagonal-constant} inner function, i.e., a
diagonal inner function, constant along the diagonal. \
This
assumption seems to be too rigid. \
To see this, we consider the following example.

\medskip

\begin{ex}\label{ex2.2}
Let $b_\alpha:=\frac{z-\alpha}{1-\overline\alpha z}$
($\alpha\in\mathbb D$), let $\theta$ be an inner function which is
coprime with $b_\alpha$ and $b_\beta$ ($\alpha\ne \beta$) and
consider the following matrix-valued function
\begin{equation}\label{2.2}
\Phi:=\begin{bmatrix} \overline\theta & \overline{\theta}\overline b_\alpha+c
\theta b_\beta\\ \overline b_\beta + c \theta^2b_\alpha &
\overline\theta\end{bmatrix}\,, \quad\hbox{where $c\in\mathbb R$
with $c\ge \left|\left| \begin{bmatrix}
                       \theta& b_\beta\\ \theta b_\alpha&\theta\end{bmatrix}\right|\right|_\infty$}\,.
\end{equation}
Then
$$
\Phi_+=c \begin{bmatrix} 0&\theta b_\beta\\
\theta^2b_\alpha&0\end{bmatrix} \quad\hbox{and}\quad
\Phi_-=\begin{bmatrix} \theta&b_\beta\\ \theta
b_\alpha&\theta\end{bmatrix}.
$$
A straightforward calculation shows that $\Phi^*\Phi=\Phi\Phi^*$. \
If
$K:=\frac{1}{c}\begin{bmatrix} \theta&b_\beta\\ \theta
b_\alpha&\theta\end{bmatrix}$, then
$$
||K||_\infty \le 1\quad\hbox{and}\quad \Phi_-^*=K\Phi_+^*,
$$
which implies that by Lemma \ref{lem1.2}, $T_\Phi$ is hyponormal. \
But a direct calculation shows that $T_\Phi$ is not normal. \
On the other hand, we observe
$$
\widetilde{\Phi_-}=\begin{bmatrix} \widetilde \theta &\widetilde\theta\widetilde{b_\alpha}\\
              \widetilde{b_\beta}&\widetilde\theta\end{bmatrix}\,,
$$
so that
$$
\widetilde{\Phi_-}^*\begin{bmatrix} f\\ g\end{bmatrix}
               = \begin{bmatrix} \overline{\widetilde \theta} &\overline{\widetilde{b_\beta}}\\
                     \overline{\widetilde\theta} \overline{\widetilde{b_\alpha}}&\overline{\widetilde\theta}\end{bmatrix}
                         \begin{bmatrix} f\\ g\end{bmatrix}\in H^2_{\mathbb C^2}
             \Longleftrightarrow \begin{cases} \overline{\widetilde\theta}f+\overline{\widetilde{b_\beta}}g \in H^2\\
                                        \overline{\widetilde\theta} \overline{\widetilde{b_\alpha}}f +
                                            \overline{\widetilde\theta}g \in H^2 \end{cases},
$$
which implies (by using the assumption that $\theta$ is coprime with
$b_\alpha$ and $b_\beta$),
$$
f\in \widetilde\theta \widetilde{b_\alpha} H^2\quad\hbox{and}\quad
                                          g\in \widetilde\theta \widetilde{b_\beta} H^2 \,.
$$
Thus we have
$$
\hbox{ker}\, H_{\widetilde{\Phi_-}^*}=\begin{bmatrix}\widetilde\theta \widetilde{b_\alpha} &0\\
                      0&\widetilde\theta \widetilde{b_\beta}\end{bmatrix}\, H^2_{\mathbb C^2}\,,
$$
so that, by Lemma \ref{lem1.1}, we can get
$$
\widetilde{\Phi_-} =\begin{bmatrix} \widetilde \theta &\widetilde\theta\widetilde{b_\alpha}\\
              \widetilde{b_\beta}&\widetilde\theta\end{bmatrix}
              =\begin{bmatrix}\widetilde\theta \widetilde{b_\alpha} &0\\
                      0&\widetilde\theta \widetilde{b_\beta}\end{bmatrix}
               \begin{bmatrix} \widetilde{b_\alpha}&\widetilde\theta\\
                                     1&\widetilde{b_\beta}\end{bmatrix}^*
                \equiv \widetilde{\Theta} \widetilde B^*
       \quad\hbox{(right coprime factorization)}\,,
$$
where
$$
\widetilde\Theta:=\begin{bmatrix}\widetilde\theta \widetilde{b_\alpha} &0\\
                      0&\widetilde\theta \widetilde{b_\beta}\end{bmatrix}\quad\hbox{and}\quad
\widetilde B:=  \begin{bmatrix} \widetilde{b_\alpha}&\widetilde\theta\\
                                     1&\widetilde{b_\beta}\end{bmatrix}.
$$
Hence we get
$$
\Phi_-= B^*\Theta=\begin{bmatrix} b_\alpha &1\\ \theta&
b_\beta\end{bmatrix}^*
                       \begin{bmatrix} \theta b_\alpha &0\\ 0&\theta b_\beta\end{bmatrix}
                         \quad\hbox{(left coprime factorization)}.
$$
But since $\Theta\equiv \begin{bmatrix} \theta b_\alpha&0\\ 0&
\theta b_\beta\end{bmatrix}$ is not diagonal-constant
we cannot apply Theorem \ref{thm2.1} to determine whether
or not $T_\Phi$ is subnormal. \
However, as we will see in the sequel,
we can conclude (using Theorem \ref{thm3.7} below) that
$T_\Phi$ is not subnormal.
\hfill$\square$
\end{ex}

\medskip

%
%

\section{Abrahamse's Theorem for matrix-valued symbols}

\noindent
Recall the representation (\ref{1.9}), and for $\Psi\in
L^\infty_{M_n}$ such that $\Psi^*$ is of bounded type, write
$\Psi=\Theta_2B^*=B^*\Theta_2$. \
Let $\Omega$ be the greatest
common left inner divisor of $B$ and $\Theta_2$. Then $B=\Omega
B_{\ell}$ and $\Theta_2=\Omega \Omega_2$ for some $B_\ell \in
H^{2}_{M_n}$ and some inner matrix $\Omega_2$. \
Therefore we can write
\begin{equation}\label{3.1}
\Psi={B^*_{\ell}}\Omega_2,\quad\hbox{where $B_\ell$ and $\Omega_2$
are left coprime:}
\end{equation}
in this case, $B_\ell^*\Omega_2$ is called a {\it left coprime
factorization} of $\Psi$. \
Similarly,
\begin{equation}\label{3.2}
\Psi=\Delta_2{B^*_{r}}, \quad\hbox{where $B_r$ and $\Delta_2$ are
right coprime:}
\end{equation}
in this case, $\Delta_2B_r^*$ is called a {\it right coprime
factorization} of $\Psi$.

\bigskip

To prove our main results, we need several auxiliary lemmas.

We begin with:
\medskip

\noindent
\begin{lem}\label{lem3.1}
(a) Let $\Phi = \Phi_-^*+ \Phi_+ \in L^{\infty}_{M_n}$ be such that
$\Phi$ and $\Phi^*$ are of bounded type. \
Then in view of (\ref{1.9}), we may write
$$
\Phi_+ =A^* \Theta_1 \quad \hbox{and} \quad \Phi_- = B^*\Theta_2\,,
$$
where $\Theta_i:=\theta_i I_n$  with an inner function $\theta_i \
(i=1,2)$. \
If $T_{\Phi}$ is hyponormal, then $\Theta_2$ is a right
inner divisor of $\Theta_1$.

(b) In view of (\ref{3.2}), $\Phi\in L^{\infty}_{M_n}$ may be
written as
$$
\Phi_+ =\Delta_1 A_r^*\ \hbox{\rm (right coprime factorization)}
\quad \hbox{and} \quad \Phi_- = \Delta_2 B_r^*\ \hbox{\rm (right
coprime factorization)}.
$$
If $T_{\Phi}$ is hyponormal, then $\Delta_2$ is a left inner divisor
of $\Delta_1$.
\end{lem}

\begin{proof}
See [CHL2, Lemmas 3.1 and 3.2].
\end{proof}

\bigskip

In the sequel, when we consider the symbol $\Phi=\Phi_-^* +
\Phi_+\in L^{\infty}_{M_n}$, which is such that $\Phi$ and $\Phi^*$
are of bounded type and for which $T_\Phi$ is hyponormal, we will,
in view of Lemma 3.1, assume that
\begin{equation}\label{5-0}
\Phi_+=  A^* \Omega_1\Omega_2\quad \hbox{and} \quad \Phi_-=
B_\ell^*\Omega_2\ \hbox{(left coprime factorization)},
\end{equation}
where $\Omega_1 \Omega_2=\Theta=\theta I_n$. We also note that
$\Omega_2 \Omega_1=\Theta$: indeed, if
$\Omega_1\Omega_2=\Theta=\theta I_n$, then $(\overline\theta I_n
\Omega_1)\Omega_2=I_n$, so that $\Omega_1(\overline\theta I_n
\Omega_2)=I_n$, which implies that $(\overline\theta I_n
\Omega_2)\Omega_1=I_n$, and hence $\Omega_2\Omega_1=\theta
I_n=\Theta$.

\bigskip


We recall the inner-outer factorization of vector-valued functions. \
If $D$ and $E$ are Hilbert spaces and if $F$ is a function with
values in $\mathcal{B}(E,D)$ such that $F(\cdot)e\in H^2_{D}(\mathbb
T)$ for each $e\in E$, then $F$ is called a strong $H^2$-function. \
The strong $H^2$-function $F$ is called an {\it inner} function if
$F(\cdot)$ is an isometric operator from $D$ into $E$. \
Write
$\mathcal{P}_{E}$ for the set of all polynomials with values in $E$,
i.e., $p(\zeta)=\sum_{k=0}^n \widehat{p}(k)\zeta^k$, $\widehat{p}(k)\in E$. \
Then the function $Fp=\sum_{k=0}^n
F\widehat{p}(k) z^k$ belongs to $H^2_{D}(\mathbb T)$. \
The strong
$H^2$-function $F$ is called {\it outer} if
$$
\hbox{cl}\, F\cdot\mathcal{P}_E=H^2_{D}(\mathbb T).
$$
Note that every $F\in H^2_{M_n}$ is a strong $H^2$-function. \
We then have an analogue of the scalar Inner-Outer Factorization
Theorem.
\bigskip

\noindent{\bf Inner-Outer Factorization.} (cf. [Ni])\quad Every
strong $H^2$-function $F$ with values in $\mathcal{B}(E, D)$ can be
expressed in the form
$$
F=F^iF^e,
$$
where $F^e$ is an outer function with values in $\mathcal{B}(E,
D^\prime)$ and $F^i$ is an inner function with values in
$\mathcal{B}(D^\prime,D)$ for some Hilbert space $D^\prime$.

\bigskip


We introduce a key idea which provides a connection
between left coprime-ness and right coprime-ness.

\medskip

\begin{df}\label{def5.1}
If $\Delta\in H^\infty_{M_n}$ is an inner function, we define
$$
D(\Delta):= \hbox{GCD}\,\bigl\{\theta I_n:\ \hbox{$\theta$ is inner
and $\Delta$ is a (left) inner divisor of $\theta I_n$} \bigr\}\,,
$$
where $\hbox{GCD}\,(\cdot)$ denotes the greatest common inner
divisor.
\end{df}

\bigskip

\begin{lem}\label{lem5.1}
If $\Delta\in H^\infty_{M_n}$ is an inner function then
\begin{equation}\label{5.1}
D(\Delta)=\delta I_n\quad\hbox{for some inner funtion $\delta$}.
\end{equation}
\end{lem}

\begin{proof}
Let $\Delta\in H^\infty_{M_n}$ be inner. \
Then since $\Delta^*$ is
evidently of bounded type, we can write, in view of (\ref{1.9}),
$$
\Delta=\Theta A^*\quad\hbox{with $\Theta\equiv\theta I_n$ for an
inner function $\theta$ and $A\in H^2_{M_n}$}.
$$
But since $\Delta $ is inner it follows that $A^*A=I_n$, so that
$\Delta A=\Theta$. \
This says that $\Delta$ is a left inner divisor of $\theta I_n$. \
Thus $D(\Delta)$ always exists for each inner
function $\Delta\in H^\infty_{M_n}$. \
For (\ref{5.1}), we observe that for any index set $I$,
\begin{equation}\label{5.2}
D(\Delta)H^2_{\mathbb C^n}=\bigvee_{i\in I} (\theta_i
I_n)H^2_{\mathbb C^n}=\bigoplus_{j=1}^n \bigvee_{i\in I} \theta_i
H^2 =\bigoplus_{j=1}^n \hbox{GCD}\,\bigl\{\theta_i: i\in I \bigr\}
H^2,
\end{equation}
which implies that $D(\Delta)=\delta I_n$ with
$\delta:=\hbox{GCD}\,\bigl\{\theta_i: i\in I \bigr\}$, giving
(\ref{5.1}).
\end{proof}

\bigskip

Note that $D(\Delta)$ is unique up to a diagonal-constant inner function
of the form $e^{i\xi} I_n$.

\bigskip

If one of two inner functions is diagonal-constant then the ``left"
coprime-ness  and the ``right" coprime-ness between them coincide.

\medskip

\begin{lem}\label{lem5.2}
Let $\Delta\in H^\infty_{M_n}$ be inner and $\Theta:=\theta I_n$ for
some inner function $\theta$. \
Then the following are equivalent:

\medskip

{\rm (a)} $\Theta$ and $\Delta$ are left coprime;

{\rm (b)} $\Theta$ and $\Delta$ are right coprime;

{\rm (c)} $\Theta$ and $D(\Delta)$ are coprime.
\end{lem}

\begin{proof}
We first prove the equivalence (b) $\Leftrightarrow$ (c).

(c) $\Rightarrow$ (b): Evident.

(b) $\Rightarrow$ (c): If $\Delta$ is a diagonal-constant inner
function then this is trivial. \
Thus we suppose that $\Delta$ is not diagonal-constant. Write
$$
D(\Delta):=\delta I_n\quad\hbox{and}\quad
D(\Delta)=\Delta\Delta_0=\Delta_0\Delta\quad\hbox{for a nonconstant
inner function $\Delta_0$}.
$$
Suppose $\Theta$ and $D(\Delta)$  are not coprime. \
Then $\theta$ and
$\delta$ are not coprime. \
Put
$$
\omega:=\hbox{GCD}\,(\theta,\ \delta)\quad\hbox{and}\quad
\Omega:=\omega I_n.
$$
Thus
$$
\Theta=\Omega\Theta_1\quad\hbox{and}\quad D(\Delta)=\Omega \Delta_1,
$$
where $\Theta_1=\theta_1 I_n$ and $\Delta_1=\delta_1 I_n$ for some
inner functions $\theta_1$ and $\delta_1$. \
Then
\begin{equation}\label{5.3}
\theta I_n=\Omega\Theta_1\quad\hbox{and}\quad \Delta\Delta_0=\delta
I_n=\Omega\Delta_1.
\end{equation}
If $\delta=\omega$ then $\delta I_n$ is an inner divisor of $\theta
I_n$, \ so that, evidently, $\Theta$ and $\Delta$ are not right coprime. \
We now suppose $\delta\ne \omega$. \
We then claim that
\begin{equation}\label{5.4}
\hbox{$\Delta$ and $\Omega$ are not right coprime.}
\end{equation}
For (\ref{5.4}), we assume to the contrary that $\Delta$ and
$\Omega$ are right coprime. \
Since by (\ref{5.3}),
$$
\delta_1 I_n=\Delta\Delta_0 \overline\omega
I_n=\Delta(\overline\omega I_n \Delta_0),
$$
it follows that $\omega I_n$ is an inner divisor of $\Delta_0$, so
that $\overline\omega I_n \Delta_0$ is inner. \
Consequently,
$$
(\delta\overline\omega)I_n=\delta_1 I_n=\Delta(\overline\omega I_n
\Delta_0),
$$
which contradicts to the definition of $D(\Delta)$. \
This proves (\ref{5.4}). \
But since $\Theta=\Omega\Theta_1=\Theta_1\Omega$, it
follows that $\Theta$ and $\Delta$ are not right coprime.

\medskip

(a) $\Leftrightarrow$ (c). Since
$\widetilde{D(\Delta)}=D(\widetilde\Delta)$, it follows from the
equivalence (b) $\Leftrightarrow$ (c) that
$$
\begin{aligned}
\hbox{$\Theta$ and $D(\Delta)$ are coprime}\
&\Longleftrightarrow\ \hbox{$\widetilde \Theta$ and $\widetilde{D(\Delta)}$ are coprime}\\
&\Longleftrightarrow\ \hbox{$\widetilde \Theta$ and $\widetilde \Delta$ are right coprime}\\
&\Longleftrightarrow\ \hbox{$\Theta$ and $\Delta$ are left coprime}.
\end{aligned}
$$
This completes the proof.
\end{proof}

\bigskip

\begin{lem}\label{lem5.3}
Let $A \in H_{M_n}^2$ be such that $\hbox{det}\, A$ is not
identically zero and $\Theta:=\theta I_n$ for some inner function $\theta$. \
Then the following are equivalent:
\medskip

{\rm (a)} $\Theta$ and $A$ are left coprime;

{\rm (b)} $\Theta$ and $A$ are right coprime.
\end{lem}

\begin{proof}
Since  $\hbox{det}\, A$ is not identically zero, the left and the
right inner divisors of $A$ are square. \
Thus we have the following
inner-outer factorizations of $A$ of the form
$$
A=A^i\, A^e=B^e\, B^i\,,
$$
where  $A^i, B^i\in H^2_{M_n}$ are inner and $A^e, B^e\in H^2_{M_n}$ are outer. \
We will show that
\begin{equation}\label{5.5}
D(A^i)=D(B^i).
\end{equation}
Write
$$
D(B^i)=B^i\Delta_0\quad\hbox{for some inner function $\Delta_0$}.
$$
Then we have
$$
\begin{aligned}
D(B^i)H^2_{\mathbb C^n} &=D(B^i)\Bigl[\hbox{cl}\, B^e
\mathcal{P}_{\mathbb C^n}\Bigr]
   =\hbox{cl}\, B^e \Bigl[D(B^i) \mathcal{P}_{\mathbb C^n}\Bigr]
      =\hbox{cl}\, B^e B^i \Bigl[\Delta_0 \mathcal{P}_{\mathbb C^n}\Bigr]\\
&=\hbox{cl}\, A^i A^e \Bigl[\Delta_0 \mathcal{P}_{\mathbb C^n}\Bigr]
     =A^i \Bigl[ \hbox{cl}\, A^e \Delta_0 \mathcal{P}_{\mathbb C^n}\Bigr]
       \subseteq A^i H^2_{\mathbb C^n},
\end{aligned}
$$
which proves that $D(B^i)H_{\mathbb C^n}^2 \subseteq  A^i H_{\mathbb
C^n}^2$, so that $A^i$ is a left inner divisor of $D(B^i)$. \
Thus by
the definition of $D(A^i)$, $D(A^i)$ is a (left) inner divisor of
$D(B^i)$. \
Similarly, we can show that $D(B^i)$ is an inner divisor
of $D(A^i)$, and hence $D(A^i)=D(B^i)$. \
Thus by Lemma \ref{lem5.2},
we have
$$
\begin{aligned}
\Theta \ \hbox{and} \ A  \ \hbox{are left coprime}
&\Longleftrightarrow \Theta \ \hbox{and} \ D(A^i)  \ \hbox{are coprime}\\
&\Longleftrightarrow \Theta \ \hbox{and} \ D(B^i)  \ \hbox{are coprime}\\
&\Longleftrightarrow\Theta \ \hbox{and} \ A  \ \hbox{are right
coprime}.
\end{aligned}
$$
\end{proof}

In Lemma \ref{lem5.3}, if $\theta$ is given as a finite Blaschke
product then the ``determinant" assumption may be dropped.

\medskip

\begin{lem}\label{lem5.5}
Let $A \in H_{M_n}^2$ and $\Theta:=\theta I_n$ for a finite Blaschke
product $\theta$. \
Then the following are equivalent:
\medskip

{\rm (a)} $\Theta$ and $A$ are left coprime;

{\rm (b)} $\Theta$ and $A$ are right coprime;

{\rm (c)} $A(\alpha)$ is invertible for each zero $\alpha$ of $\theta$.
\end{lem}

\begin{proof}
See [CHL2, Lemma 3.3].
\end{proof}

\begin{lem}\label{lem5.6}
Let $A \in H_{M_n}^{2}$ and $\Theta$ be a diagonal inner function
whose diagonal entries are nonconstant. \
If $Af=0$ for each $f\in\mathcal{H}_\Theta$, then $A=0$.
\end{lem}

\begin{proof}
Write $A\equiv [a_{ij}]_{1\le i,j\le n}$ ($a_{ij}\in H^2$)
and $\Theta\equiv \hbox{diag}(\theta_1,\cdots, \theta_n)$
(where $\theta_j$ is a nonconstant inner function for each $j=1,\cdots,n$). \
Suppose $Af=0$ for each $f\in\mathcal{H}_\Theta$. \
Choose an outer function $h_j$ in $\mathcal{H}_{\theta_j}$
which is invertible in $H^\infty$ (cf. [CHL2, Lemma 3.4]). \
For each $j=1,\cdots, n$, define $g_j:=(0,\cdots,0,h_j,0,\cdots,0)^t$ (where $h_j$ is
the $j$-th component). Clearly, $g_j\in \mathcal{H}_{\Theta}$. \
Thus by assumption, $Ag_j=0$ for each $j$, so that
$a_{ij}h_j=0$ for each $i,j=1,\cdots,n$. \
But since $h_j$ is invertible, $a_{ij}=0$ for each $i,j=1,\cdots,n$,
i.e., $A=0$.
\end{proof}

\bigskip


We are now ready to prove the main result of this paper.
\smallskip

\begin{thm}\label{thm3.7} {\rm (Abrahamse's Theorem for matrix-valued symbols, Version I)}
Suppose $\Phi=\Phi_-^*+ \Phi_+\in L^\infty_{M_n}$ is such that
$\Phi$ and $\Phi^*$ are of bounded type and $\hbox{det}\, \Phi_+$
and $\hbox{det}\, \Phi_-$ are not identically zero. \
Then in view of (\ref{3.1}), we may write
$$
\Phi_- = B^* \Theta \quad\hbox{(left coprime factorization)}\,.
$$
Assume that $\Theta$ is a diagonal inner matrix function (which is not
necessarily diagonal-constant) and that $\Theta$ has a
nonconstant diagonal-constant inner divisor $\Omega\equiv \omega I_n$ ($\omega$ inner)
such that $\Omega$ and $\Theta\Omega^*$ are coprime. \
If
\medskip

{\rm (i)} $T_\Phi$ is hyponormal; and

{\rm (ii)} $\text{\rm ker}\,[T_{\Phi}^*, T_{\Phi}]$ is invariant
under $T_{\Phi}$\,,
\medskip

\noindent
then $T_{\Phi}$ is either normal or analytic. \
Hence, in
particular, if $T_\Phi$ is subnormal then it is either normal or analytic.
\end{thm}

\medskip

\begin{proof}
For notational convenience, we let $\Theta_2:=\Theta$. \
In view of Lemma \ref{lem3.1}(a), we may write
$$
\Phi_+=\Theta_0\Theta_2 A^*,
$$
where $\Theta_0\Theta_2=\theta I_n$ with an inner function $\theta$ and $A\in H^2_{M_n}$. \
If $\Theta_2$
is constant then $\Phi_-\in M_n$, so that $T_\Phi$ is analytic. \
Suppose
that $\Theta_2$ is nonconstant and $\Theta_2:=\Omega \Delta$, where
$\Omega\equiv \omega I_n$ with a nonconstant inner function $\omega$,
and $\Omega$ and $\Delta$ are coprime. \

We split the proof into six steps: each step is significant as a
separate mathematical statement. \


\bigskip

STEP 1:\ We first claim that
\begin{equation}\label{3.9}
\Theta_0 H^2_{\mathbb C^n}\ \subseteq\ \hbox{\rm ker}\,[T_{\Phi}^*,
T_{\Phi}].
\end{equation}
Indeed, the inclusion (\ref{3.9}) follows from a slight extension of [CHL2,
Theorem 3.5], in which $\Theta_2$ is a diagonal inner function of
the form $\Theta_2=\theta_2 I_n$. \
In fact, a careful analysis for the
proof of [CHL2, STEP 1 of the proof of Theorem 3.5] shows that the
proof does not employ the diagonal-{\it constant}-ness of
$\Theta_2$, but uses only the diagonal-constant-ness of
$\Theta_0\Theta_2$.


\bigskip

STEP 2: We also argue that if $K\in\mathcal E(\Phi)$, then
\begin{equation}\label{3.15}
\hbox{cl ran}\, H_{A\Theta_2^*}\subseteq
\hbox{ker}\,(I-T_{\widetilde{K}}T_{\widetilde{K}}^*).
\end{equation}
To see this, we observe that if $K\in\mathcal{E}(\Phi)$ then by
(\ref{1.5-1}),
\begin{equation}\label{3.13-2}
[T_\Phi^*,
T_\Phi]=H_{\Phi_+^*}^*H_{\Phi_+^*}-H_{K\Phi_+^*}^*H_{K\Phi_+^*}
=H_{\Phi_+^*}^* (I- T_{\widetilde K}T_{\widetilde K}^*)
H_{\Phi_+^*}\,,
\end{equation}
so that
$$
\hbox{ker}\,[T_\Phi^*, T_\Phi]=\hbox{ker}\,(I- T_{\widetilde
K}T_{\widetilde K}^*) H_{\Phi_+^*}.
$$
Thus by (\ref{3.9}),
$$
\{0\}=(I-T_{\widetilde{K}}T_{\widetilde{K}}^*)H_{A\Theta_2^*\Theta_0^*}(\Theta_0
H_{\mathbb C^n}^2) =
(I-T_{\widetilde{K}}T_{\widetilde{K}}^*)H_{A\Theta_2^*}(H_{\mathbb
C^n}^2)\,,
$$
giving (\ref{3.15}).

\medskip

We note that STEP 1 and STEP 2 hold with no restriction on $\Theta\equiv \Theta_2$.


\bigskip

STEP 3: We claim that
\begin{equation}\label{3.16-8}
\hbox{$\Omega$ and $\Theta_0$ are (right) coprime.}
\end{equation}
To see this we assume to the contrary that $\Omega$ and $\Theta_0$
are not coprime. \
Since $\Theta_2\Theta_0\equiv \Omega \Delta\Theta_0$ and $\Omega$ are diagonal-constant,
it follows that $\Delta\Theta_0$ is diagonal-constant. \
Also there exists an inner function $\Delta'$ such that
$\Delta\Delta'=D(\Delta)$. \
Thus we can write
$$
\Theta_2\Theta_0=\Omega \Delta \Theta_0=\Omega \Delta \Delta'\Gamma
$$
where $\Gamma:=\gamma I_n$ for some
inner function $\gamma$. \
Since by assumption, $\Omega$ and $\Delta$ are coprime it
follows from Lemma 3.4 that $\Omega$ and $\Delta\Delta'$ are
coprime. \
Therefore  we have
$$
\Omega':= \hbox{GCD}\,\{\Omega,\ \Theta_0\}\equiv  \omega' I_n,
$$
where $\omega'=\hbox{GCD}\{\omega, \gamma\}$ is not constant. \
Thus we can write
\begin{equation}\label{3.16-6}
\Theta_0=\Omega'
\Theta_0^{\prime}=\Theta_0^{\prime}\Omega'\quad\hbox{and}\quad
\Theta_2=\Omega' \Theta_2^{\prime}=\Theta_2^{\prime}\Omega'
\end{equation}
for some diagonal inner functions
$\Theta_0^{\prime},\Theta_2^{\prime}$. \
Then since
$\Theta_2\Theta_0^{\prime}=\Omega'\Theta_2^{\prime}\Theta_0^{\prime}=
\Omega'\Theta_0^{\prime}\Theta_2^{\prime}=\Theta_0\Theta_2^{\prime}$,
it follows from (\ref{3.9}) that
\begin{equation}\label{3.16-16}
\Theta_2\Theta_0^{\prime} H_{\mathbb C^n}^2\subseteq \Theta_0
H_{\mathbb C^n}^2\subseteq \hbox{\rm ker}\,[T_{\Phi}^*, T_{\Phi}].
\end{equation}
Note that
$$
\Theta_2 \Theta_0'=\Omega \Delta \Theta_0'=\Omega \Delta
\Delta'  (\gamma \overline{\omega'})I_n \quad
(\gamma\overline{\omega'}\in H^2).
$$
Thus since $\Theta_2\Theta_0'$ is diagonal-constant and hence,
$\Theta_2^*B\Theta_2\Theta_0'\in H^2_{M_n}$, it follows that
$$
H_{\Phi_-^*}(\Theta_2 \Theta_0'H_{\mathbb C^n}^2)
=H_{\Theta_2^*B}(\Theta_2 \Theta_0'H_{\mathbb C^n}^2)=0\,.
$$
Thus by (\ref{3.16-16}), we have
$$
H_{\Phi_+^*}(\Theta_2 \Theta_0'H_{\mathbb C^n}^2)
= H_{A\Theta_2^*\Theta_0^*} (\Theta_2\Theta_0^{\prime} H_{\mathbb C^n}^2)
=\{0\},\quad \hbox{so that}\ \ H_{A\Omega'^*}(H_{\mathbb
C^n}^2)=\{0\}.
$$
Thus we must have that $G\equiv A\Omega'^* \in H_{M_n}^2$. \
Then we can write
$$
\Phi_+=\Theta_0\Theta_2 A^*=\Omega'\Theta_0^\prime\Theta_2^\prime
\Omega' A^* =\Omega'\Theta_0^\prime\Theta_2^\prime G^*,
$$
which leads to a contradiction because the representation
$\Phi_+=\Theta_0\Theta_2 A^*$ is in ``minimal" form in view of
(\ref{1.9}). \
This proves (\ref{3.16-8}).


\bigskip

STEP 4:\ We claim that
\begin{equation}\label{3.8-11}
\hbox{$A$ and $\Omega$ are left coprime.}
\end{equation}
Indeed, by assumption $B$ and $\Theta_2$ are left coprime, so we can see
that $B$ and $\Omega$ are left coprime. \
Since $\hbox{det}\,\Phi_-$ is not identically zero and hence,
$\hbox{det}\,B$ is not either, it follows from Lemma \ref{lem5.3} that
$B$ and $\Omega$ are right coprime. \
Thus by Lemma \ref{lem3.1}(b), we can write
$$
\Phi_+ = \Theta_0 \Theta_2 A^* =\Omega\Delta_1 A_r^*\,,
$$
where $A_r$ and $\Omega \Delta_1$ are right coprime. \
In particular,
since by (\ref{3.16-8}), $\Omega$ and $\Theta_0$ are right coprime,
$A$ and $\Omega$ are
right coprime. \
Since by assumption, $\hbox{det}\,\Phi_+$ is not
identically zero and hence, $\hbox{det}\,A$ is not either, it
follows again from Lemma \ref{lem5.3} that
$A$ and $\Omega$ are left coprime. \
This proves (\ref{3.8-11}).


\bigskip

STEP 5: We now claim that
\begin{equation}\label{3.13-1}
\hbox{$\mathcal{E}(\Phi)$ contains an inner function $K$.}
\end{equation}
We first observe that $A\Theta_2^*=A\Omega^*\Delta^*=\Omega^* A \Delta^*$,
so that
$$
\hbox{cl ran}\, H_{A\Theta_2^*}=\Bigl(\hbox{ker}\, H_{\widetilde
\Delta^* \widetilde A\widetilde \Omega^*}\Bigr)^{\perp}\,.
$$
Since by (\ref{3.8-11}), $A$ and $\Omega$ are left coprime (so that
$\widetilde A$ and $\widetilde \Omega$ are
right coprime), it follows that
$$
\aligned f \in \hbox{ker}\, H_{\widetilde \Delta^*\widetilde A
\widetilde \Omega^*} &\Longrightarrow
\widetilde \Delta^* \widetilde A \widetilde \Omega^* f \in H_{\mathbb C^n}^2\\
&\Longrightarrow \widetilde A \widetilde \Omega^* f \in
\widetilde{\Delta}H_{\mathbb
C^n}^2 \subseteq H_{\mathbb C^n}^2 \\
&\Longrightarrow f \in \hbox{ker}\, H_{\widetilde A \widetilde \Omega^*}\\
&\Longrightarrow f \in \widetilde{\Omega}H_{\mathbb C^n}^2\quad
\hbox{(by Lemma \ref{lem1.1})},
\endaligned
$$
which implies that
$$
\hbox{ker}\, H_{\widetilde \Delta^* \widetilde A\widetilde \Omega^*}
\subseteq \widetilde{\Omega}H_{\mathbb C^n}^2\,,
$$
and hence,
$$
\mathcal{H}_{\widetilde\Omega} \subseteq \hbox{cl ran}\,
H_{A\Theta_2^*}\,.
$$
Thus by (\ref{3.15}),
\begin{equation}\label{3.15-7}
\mathcal{H}_{\widetilde\Omega} \subseteq
\hbox{ker}\,(I-T_{\widetilde{K}}T_{\widetilde{K}}^*).
\end{equation}
We thus have
\begin{equation}\label{3.14-2}
F=T_{\widetilde{K}}T_{\widetilde{K}}^* F\quad\hbox{for each
$F\in\mathcal H_{\widetilde{\Omega}}$}\,.
\end{equation}
But since $||\widetilde K||_\infty=||K||_\infty \le 1$,
it follows from a direct calculation that
$$
||P_n(\widetilde K^*F)||_2=||\widetilde K^*F||_2,
$$
which implies $\widetilde K^*F\in H^2_{\mathbb C^n}$. \
Therefore by
(\ref{3.14-2}),
$(I-\widetilde{K}\widetilde{K}^*)F=0$ for each $F\in\mathcal
H_{\widetilde{\Omega}}$. \
Thus by Lemma \ref{lem5.6}, $K^*K=I$, which proves (\ref{3.13-1}).


\bigskip

STEP 6:
We finally claim that
$$
\hbox{$T_\Phi$ is normal.}
$$
To see this, we first observe that if $K\in\mathcal{E}(\Phi)$ is inner
then it follows from (\ref{3.15-7}) that
\begin{equation}\label{3.14-3}
\mathcal{H}_{\widetilde \Omega} \subseteq
\hbox{ker}\,(I-T_{\widetilde{K}}T_{\widetilde{K}}^*) =
\hbox{ker}\,H_{\widetilde{K}^*}^* H_{\widetilde K^*}=
\hbox{ker}\,H_{\widetilde K^*}.
\end{equation}
Write $K:=[k_{ij}]_{1\le i,j\le n}\in H^\infty_{M_n}$. \
It thus follows that for each $i,j=1,2,\cdots,n$,
$$
k_{ij}(\overline z) h\in H^2\quad\hbox{for an invertible function $h\in \mathcal{H}_{\widetilde\omega}$}.
$$
Therefore each $k_{ij}$ is constant and
hence, $K$ is constant. \
Therefore by (\ref{3.13-2}), $[T_\Phi^*,
T_\Phi]=0$, i.e., $T_\Phi$ is normal. This completes the proof.
\end{proof}

\bigskip

In Theorem \ref{thm3.7}, if $\Theta$ has a nonconstant diagonal-constant
inner divisor of the form $\omega I_n$ with a Blaschke factor
$\omega$, then we can strengthen Theorem \ref{thm3.7} by dropping
the ``determinant" assumption.

\begin{cor}\label{cor3.8}
Suppose $\Phi=\Phi_-^*+ \Phi_+\in L^\infty_{M_n}$ is such that
$\Phi$ and $\Phi^*$ are of bounded type. \
Then in view of (\ref{3.1}), we may write
$$
\Phi_- = B^* \Theta \quad\hbox{(left coprime factorization)}\,,
$$
where $\Theta$ is a diagonal inner matrix function. \
Assume that $\Theta$ has a
nonconstant diagonal-constant inner divisor $\Omega\equiv \omega I_n$ with a finite Blaschke product $\omega$
such that $\Omega$ and $\Theta\Omega^*$ are coprime. \
If
\medskip

{\rm (i)} $T_\Phi$ is hyponormal; and

{\rm (ii)} $\text{\rm ker}\,[T_{\Phi}^*, T_{\Phi}]$ is invariant
under $T_{\Phi}$\,,
\medskip

\noindent then $T_{\Phi}$ is either normal or analytic. \
Hence, in
particular, if $T_\Phi$ is subnormal then it is either normal or analytic.
\end{cor}

\begin{proof}
If we put $\Omega:=\omega I_n$ with a finite Blaschke product $\omega$, then
we may use Lemma \ref{lem5.5} in place of Lemma \ref{lem5.3}. \
Thus
we can drop the ``determinant" condition in Theorem \ref{thm3.7}
because Theorem \ref{thm3.7} employs the determinant condition only for
the equivalence of the left coprime-ness and the right coprime-ness
between $\Omega$ and some $D\in H^2_{M_n}$.
\end{proof}

\medskip

In Corollary \ref{cor3.8}, if $\Theta$ is diagonal-constant then
we may take $\Theta=\Omega$, and hence $\Theta\Omega^*=I$, so that
$\Theta$ and $\Theta\Omega^*$ are trivially coprime. \
Thus if $\Theta$ is
diagonal-constant then Corollary \ref{cor3.8} reduces to Theorem \ref{thm2.1}.

\begin{ex} (Example \ref{ex2.2} Revisited)\quad
We take a chance to reconsider the function given in (\ref{2.2}):
$$
\Phi:=\begin{bmatrix} \overline\theta & \overline{\theta}\overline b_\alpha+c
\theta b_\beta\\ \overline b_\beta + c \theta^2b_\alpha &
\overline\theta\end{bmatrix}\,, \quad\hbox{where $c\in\mathbb R$
with $c\ge \left|\left| \begin{bmatrix}
                       \theta& b_\beta\\ \theta b_\alpha&\theta\end{bmatrix}\right|\right|_\infty$}\,.
$$
In Section 2 we have shown that $T_\Phi$ is hyponormal, but not normal. \
On the other hand, we know that
$$
\Phi_-= B^*\Theta=\begin{bmatrix} b_\alpha &1\\ \theta&
b_\beta\end{bmatrix}^*
                       \begin{bmatrix} \theta b_\alpha &0\\ 0&\theta b_\beta\end{bmatrix}
                         \quad\hbox{(left coprime factorization)}\,,
$$
But since
$$
\Theta=\begin{bmatrix} \theta b_\alpha& 0\\ 0& \theta b_\beta\end{bmatrix}
=\begin{bmatrix} \theta & 0\\ 0& \theta\end{bmatrix}\,
\begin{bmatrix} b_\alpha& 0\\ 0& b_\beta\end{bmatrix}
$$
and
$\begin{bmatrix} \theta & 0\\ 0& \theta\end{bmatrix}$ and
$\begin{bmatrix} b_\alpha& 0\\ 0& b_\beta\end{bmatrix}$ are coprime
(since $\theta$ is coprime with $b_\alpha$ and $b_\beta$),
it follows from Theorem \ref{thm3.7} that
$T_\Phi$ is not subnormal.
\hfill$\square$
\end{ex}

\bigskip

\begin{rem}\label{rem3.11}
The assumption ``$\Theta$ is diagonal" in Theorem \ref{thm3.7}
seems to be still somewhat rigid. \
A careful analysis of the proof of Theorem \ref{thm3.7}
shows that this assumption was used only in proving STEP 3 (and whence STEP 4). \
However, we did not directly employ the assumption ``$\Theta\equiv \Theta_2$ is diagonal"
in the proofs of STEP 5 and STEP 6; instead we used the statement in STEP 4. \
Also, we have already recognized that STEP 1 and STEP 2 hold with no
restriction on $\Theta$. \
Therefore if we make the assumption
``$A$ and $\Omega$ are left coprime" in Theorem \ref{thm3.7}, then
Theorem \ref{thm3.7} still holds for a general form of $\Theta$. \
Moreover, if we assume that $A$ and $\Theta$ are left coprime then
we do not need an additional assumption that
$\Omega$ and $\Theta\Omega^*$ are coprime because
it was used only in the proof of STEP 3 (as an auxiliary lemma for STEP 4). \
Consequently,
if we strengthen the left coprime-ness for the analytic part of the symbol
then we can relax the restriction on $\Theta$. \
\end{rem}

Therefore we get:
\medskip

\begin{cor}\label{cor3.11} {\rm (Abrahamse's Theorem for matrix-valued symbols, Version II)}
\quad
Suppose $\Phi=\Phi_-^*+ \Phi_+\in L^\infty_{M_n}$ is such that
$\Phi$ and $\Phi^*$ are of bounded type and $\hbox{det}\, \Phi_+$
and $\hbox{det}\, \Phi_-$ are not identically zero. \
Then in view of (\ref{5-0}), we may write
$$
\Phi_+=A^*\Theta_0\Theta_2\quad\hbox{and}\quad \Phi_- = B^* \Theta_2 \,,
$$
where
$\Theta_0\Theta_2=\theta I_n$ with an inner function $\theta$. \
Assume that $A,B$ and $\Theta_2$ are left coprime and
$\Theta_2$ has a nonconstant diagonal-constant inner divisor $\Omega\equiv \omega I_n$ ($\omega$ inner). \
If
\medskip

{\rm (i)} $T_\Phi$ is hyponormal; and

{\rm (ii)} $\text{\rm ker}\,[T_{\Phi}^*, T_{\Phi}]$ is invariant
under $T_{\Phi}$\,,
\medskip

\noindent then $T_{\Phi}$ is either normal or analytic. \
Hence, in
particular, if $T_\Phi$ is subnormal then it is either normal or analytic.
\end{cor}

\begin{proof}
This follows from Remark \ref{rem3.11} and an analysis of the proof
of Theorem \ref{thm3.7}.
\end{proof}

\noindent
\begin{rem}\label{rem3.13}
Observe that Corollary \ref{cor3.11} is a substantive generalization
of \cite[Theorem 3.5]{CHL2}, in which $\Theta_2$ is
diagonal-constant.
\end{rem}

%
%
%
%

\section{A Subnormal Toeplitz Completion }

\medskip
%
%

Given a partially specified operator matrix with some known entries,
the problem of finding suitable operators to complete the given
partial operator matrix so that the resulting matrix satisfies
certain given properties is called a {\it completion problem}. \ A
{\it subnormal completion} of a partial operator matrix is a
particular specification of the unspecified entries resulting in a
subnormal operator. \
A {\it partial block Toeplitz matrix} is simply
an $n\times n$ matrix some of whose entries are specified Toeplitz
operators and whose remaining entries are unspecified. \
A {\it
subnormal Toeplitz completion} of a partial block Toeplitz matrix is
a subnormal completion whose unspecified entries are Toeplitz
operators. \

In \cite{CHL1}, the following subnormal Toeplitz completion problem
was considered:

\medskip

\noindent{\bf Problem A.} Let $U$ be the unilateral shift on $H^2$. \  Complete
the unspecified Toeplitz entries of the partial block Toeplitz
matrix $A:=\left[\begin{smallmatrix} U^*& ?\\
?&U^*\end{smallmatrix}\right]$ to make $A$ subnormal. \

\medskip

The solution of Problem A given in \cite[Theorem 5.1]{CHL1} relies
upon very intricate and long computations using the symbol involved. \
In this section,  by employing our main result in Section 3, we
provide a shorter and more insightful proof for the following
problem which is a more general version of Problem A:\

\medskip

\noindent{\bf Problem B.}\ Let $b_\lambda$ be a Blaschke factor of
the form $b_\lambda(z):=\frac{z-\lambda}{1-\overline \lambda z}$
($\lambda\in \mathbb D$). \
Complete the unspecified Toeplitz entries
of the partial block Toeplitz matrix
$$
A:=\begin{bmatrix} T_{\overline b_\alpha} & ?\\ ?& T_{\overline
b_\beta}\end{bmatrix}\quad\hbox{($\alpha,\beta\in\mathbb D$)}
$$
to make $A$ subnormal.

\bigskip

To answer Problem B, we need:


\begin{lem}\label{lem4.1}\quad
Let
$$
\Phi_-=\begin{bmatrix} b_\alpha &\theta_1 \overline{b}\\
\theta_0 \overline{a}& b_\alpha \end{bmatrix} \quad (a \in \mathcal
H_{z\theta_0},\ b\in \mathcal H_{z\theta_1}\ \hbox{and $\theta_j$
inner} \  (j=0,1))
$$
and $\hbox{\rm ker}\,H_{\Phi_-^*}=\Delta H^2_{\mathbb C^2}$. \

\medskip

(a) If  $\theta_0=b_\alpha^n \theta_0^{\prime}$ {\rm ($n\ge 1$,
$\theta_0'(\alpha)\ne 0$)}
   and $\theta_1(\alpha) \neq 0$, then
$$
\Delta=
\begin{cases}\qquad\qquad
\begin{bmatrix} b_\alpha \theta_1&0\\0&\theta_0\end{bmatrix} \quad &(n=1);\\
\frac{1}{\sqrt{|\gamma|^2+1}}\begin{bmatrix} b_\alpha
\theta_1&\gamma
\theta_1\\-\overline{\gamma}\theta_0&b_\alpha^{n-1}\theta_0'\end{bmatrix}
\quad &(n \geq 2)\quad
\left(\gamma:=-\frac{a(\alpha)}{\theta_1(\alpha)}\right).
\end{cases}
$$

(b) If  $\theta_1=b_\alpha^n \theta_1^{\prime}$ {\rm ($n\ge 1$,
$\theta_1'(\alpha)\ne 0$)} and $\theta_0(\alpha) \neq 0$, then
$$
\Delta=
\begin{cases}\qquad\qquad
\begin{bmatrix}\theta_1&0\\0& b_\alpha\theta_0\end{bmatrix} \quad &(n=1);\\
\frac{1}{\sqrt{|\gamma|^2+1}}\begin{bmatrix}b_\alpha^{n-1}\theta_1'&
-\overline{\gamma}\theta_1\\ \gamma\theta_0&
b_\alpha\theta_0\end{bmatrix} \quad &(n \geq 2)\quad
\left(\gamma:=-\frac{b(\alpha)}{\theta_0(\alpha)}\right).
\end{cases}
$$

(c) If  $\theta_0(\alpha)\ne 0$ and $\theta_1(\alpha)\ne 0$, then
$$
\Delta=
\begin{bmatrix} b_\alpha \theta_1&0\\0& b_\alpha\theta_0\end{bmatrix}\,.
$$

(d) If $\theta_0= b_\alpha \theta_0'$ and $\theta_1 = b_\alpha
\theta_1'$ then
$$
\Delta=
\begin{cases}
\qquad\qquad
   \begin{bmatrix} \theta_1 &0\\0&\theta_0\end{bmatrix}
&\Bigl((ab)(\alpha)\neq (\theta_0'\theta_1')(\alpha)\Bigr);\\
      \frac{1}{\sqrt{|\gamma|^2+1}}
         \begin{bmatrix} \theta_1&\gamma\theta_1^{\prime}\\
              -\overline{\gamma}\theta_0&\theta_0^{\prime}\end{bmatrix} \quad
&\Bigl((ab)(\alpha)=(\theta_0'\theta_1')(\alpha)\Bigr) \quad
\left(\gamma:=-\frac{a(\alpha)}{\theta_1'(\alpha)}\right).
\end{cases}
$$
\end{lem}

\begin{proof}
This follows from a slight variation of the proof of [CHL1, Lemmas
5.4, 5.5, and 5.6].
\end{proof}

\bigskip

We are ready for:

\medskip

\begin{thm} \label{thm4.2}
Let $\varphi, \psi \in L^{\infty}$ and consider
$$
A:=\begin{bmatrix} T_{\overline b_\alpha} & T_\varphi\\ T_\psi&
T_{\overline
b_\beta}\end{bmatrix}\quad\hbox{($\alpha,\beta\in\mathbb D$)}\,,
$$
where $b_\lambda$ is a Blaschke factor of the form
$b_\lambda(z):=\frac{z-\lambda}{1-\overline \lambda z}$ ($\lambda\in
\mathbb D$). \
The following statements are equivalent.
\begin{itemize}
\item[(a)] $A$ is normal.
\item[(b)] $A$ is subnormal.
\item[(c)] $A$ is $2$-hyponormal.
\item[(d)] $\alpha=\beta$ and one of the following conditions
holds:
\end{itemize}
\medskip
\begin{itemize}
\item[1.] $\varphi=e^{i\theta} b_\alpha + \zeta$\quad and\quad $\psi=e^{i\omega}\varphi$\quad
\hbox{\rm ($\zeta\in\mathbb C$; $\theta,\omega\in [0,2\pi)$);}
\item[2.] $\varphi=\mu\, \overline b_\alpha + e^{i\theta}\sqrt{1+|\mu|^2}\,b_\alpha + \zeta$\quad and\quad
$\psi=e^{i\,(\pi-2\,{\rm arg}\,\mu)}\varphi$ \ ($\mu,\zeta\in\mathbb
C$, $\mu\ne 0$, $|\mu|\ne 1$, $\theta\in [0,2\pi)$),
\end{itemize}
\medskip
\noindent
except in the following special case:
\begin{equation}\label{6.6-0}
\hbox{$\varphi_-=b_\alpha \theta_0^{\prime} \overline a$ and
$\psi_-=b_\alpha \theta_1^{\prime} \overline b$ (coprime
factorizations)\ \  with
$(ab)(\alpha)=(\theta_0^{\prime} \theta_1^{\prime})(\alpha)\ne 0$\,.}
\end{equation}
However, if we also know that $\varphi,\psi\in L^\infty$ are rational functions
having the same number of poles then
either {\rm (2)} holds for $|\mu|=1$ or
$$
\varphi=e^{i\theta}\overline b_\alpha  + 2 e^{i\omega} b_\alpha +
\zeta \quad\hbox{and}\quad \psi=e^{-2i\theta}\varphi\quad
(\theta,\omega\in [0,2\pi),\ \zeta\in\mathbb C):
$$
in this case, $A+e^{-i\theta} \zeta$ is quasinormal.
\end{thm}
\bigskip

As a straightforward consequence of Theorem \ref{thm4.2}, we obtain

\medskip

\begin{cor}\label{cor4.3}
Let
$$
A:=\begin{bmatrix} U^* & U^*+2U\\ U^*+2U & U^*\end{bmatrix}\,,
$$
where $U\equiv T_z$ is the unilateral shift on $H^2$. \
Then $A$ is a quasinormal (therefore subnormal) completion of
$\left[\begin{smallmatrix} U^*& ?\\
?&U^*\end{smallmatrix}\right]$, and $A$ is not normal.
\end{cor}

\bigskip

\begin{proof}[Proof of Theorem \ref{thm4.2}] \
Clearly (a) $\Rightarrow$ (b) and (b) $\Rightarrow$ (c). \
Moreover, (d) $\Rightarrow$ (a) follows from a straightforward calculation. \

(c) $\Rightarrow$ (d): \ Write
$$
\Phi\equiv \left[\begin{matrix} \overline b_\alpha&\varphi\\
\psi&\overline b_\beta\end{matrix}\right] \equiv \Phi_-^*+\Phi_+
=\begin{bmatrix} b_\alpha& \psi_-\\
\varphi_-& b_\beta\end{bmatrix}^{\ast}+\begin{bmatrix} 0&\varphi_+\\
\psi_+& 0\end{bmatrix}
$$
and assume that $T_\Phi$ is $2$-hyponormal. \
Since $\hbox{ker}\,
[T^*, T]$ is invariant under $T$ for every $2$-hyponormal operator
$T \in \mathcal{B}(\mathcal{H})$, we note that Theorem \ref{thm3.7}
holds for $2$-hyponormal operators $T_{\Phi}$
under the same assumption on the symbol. \
We claim that
\medskip
\begin{align}
&|\varphi|=|\psi|, \; \textrm{and} \label{6.5}\\
&\Phi \ \hbox{and} \ \Phi^* \ \hbox{are of bounded type.}
\label{6.6}
\end{align}

\noindent
Indeed, if $T_\Phi$ is hyponormal then $\Phi$ is normal,
so that a straightforward calculation gives (\ref{6.5}). \
Also, by Lemma \ref{lem1.2}
there exists a matrix function
$K\equiv \left[\begin{smallmatrix} k_1&k_2\\
k_3&k_4\end{smallmatrix}\right] \in \mathcal{E}(\Phi)$, i.e.,
$||K||_\infty\le 1$ such that $\Phi-K\Phi^*\in H^\infty_{M_2}$,
i.e.,
\begin{equation}\label{6.6-1}
\left[\begin{matrix} \overline b_\alpha&\overline{\varphi_-}\\
\overline{\psi_-}& \overline b_\beta\end{matrix}\right]\, -\, \left[\begin{matrix} k_1&k_2\\
k_3&k_4\end{matrix}\right]\, \left[\begin{matrix} 0&\overline {\psi_+}\\
\overline {\varphi_+}&0\end{matrix}\right] \in H^2_{M_2},
\end{equation}
which implies that
$$
H_{\overline b_\alpha}=H_{k_2\overline {\varphi_+}}=H_{\overline
{\varphi_+}}T_{k_2} \quad\hbox{and}\quad H_{\overline
b_\beta}=H_{k_3\overline{\psi_+}}=H_{\overline {\psi_+}}T_{k_3}.
$$
If $\overline{\varphi_+}$ is not of bounded type then
$\hbox{ker}\,H_{\overline{\varphi_+}}=\{0\}$, so that $k_2=0$, a
contradiction; and if $\overline{\psi_+}$ is not of  bounded type
then $\hbox{ker}\,H_{\overline{\psi_+}}=\{0\}$, so that $k_3=0$, a
contradiction. \
Thus $\overline{\varphi_+}$ and
$\overline{\psi_+}$ are of bounded type, so that $\Phi^*$ is of
bounded type. \
Since $T_{\Phi}$ is hyponormal, it follows from (\ref{1.3-3}) that
$\Phi$ is also of
bounded type, giving (\ref{6.6}). \
Thus we can write
$$
\varphi_-:=\theta_0\overline{a}\quad\hbox{and}\quad \psi_-:=\theta_1
\overline{b}\quad\hbox{($a\in \mathcal H_{z\theta_0}$, $b\in \mathcal
H_{z\theta_1}$)},
$$
where $\theta_0$ and $\theta_1$ are inner, $a$ and $\theta_0$ are
coprime and $b$ and $\theta_1$ are coprime. \
On the other hand,
by (\ref{6.6-1}), we have
\begin{equation}\label{6.7-1}
\begin{cases}
\overline b_\alpha-k_2 \overline{\varphi_+}\in H^2,\quad
\overline{\theta}_1 b-k_4\overline{\varphi_+}\in H^2\\
\overline b_\beta-k_3 \overline{\psi_+}\in H^2,\quad
\overline{\theta}_0 a-k_1\overline{\psi_+}\in H^2,
\end{cases}
\end{equation}
which implies that the following Toeplitz
operators are all hyponormal (by Cowen's Theorem):
\begin{equation}\label{6.7-5}
T_{\overline b_\alpha+\varphi_+},\ \ T_{\overline{\theta}_1 b +
\varphi_+},\ \ T_{\overline b_\beta+\psi_+},\ \
T_{\overline{\theta}_0 a + \psi_+}.
\end{equation}
Then by the scalar-valued version of Lemma \ref{lem3.1}, we can
write
\begin{equation}\label{6.8-1}
\varphi_+=\theta_1\theta_3\overline d\quad\hbox{and}\quad
\psi_+=\theta_0\theta_2\overline c\quad\hbox{($d\in\mathcal
H_{z\theta_1\theta_3}$, $c\in\mathcal H_{z\theta_0\theta_2}$)},
\end{equation}
where $\theta_2$ and $\theta_3$ are inner, $d$ and
$\theta_1\theta_3$ are coprime, and $c$ and $\theta_0\theta_2$ are
coprime. \
In particular, $d(\alpha)\ne 0$ and $c(\beta)\ne 0$. \
We now claim that
\begin{equation}\label{6.8-2}
\alpha=\beta.
\end{equation}
Assume to the contrary that $\alpha\ne\beta$. Since $\Phi$ is
normal, i.e., $\Phi\Phi^*=\Phi^*\Phi$, we have
$$
\begin{bmatrix}
    \overline b_{\alpha} & \varphi\\ \psi&\overline{b}_{\beta}\end{bmatrix}
      \begin{bmatrix} {b_{\alpha}}& \overline{\psi}\\ \overline{\varphi}&{b}_{\beta} \end{bmatrix}
         =\begin{bmatrix} {b_{\alpha}}& \overline{\psi}\\ \overline{\varphi}&{b}_{\beta} \end{bmatrix}
             \begin{bmatrix} \overline b_{\alpha} & \varphi\\ \psi&\overline{b}_{\beta} \end{bmatrix}\,,
$$
which gives
$$
\overline b_{\alpha} \overline{\psi}+\varphi b_{\beta}
     =b_{\alpha}\varphi+\overline{\psi}\overline{b}_{\beta},\ \ \hbox{i.e.,}\ \
         (b_{\alpha} -  b_{\beta}) (\psi+ \overline b_\alpha \overline b_\beta \overline\varphi)=0\,,
$$
which implies that $\psi=-\overline b_{\alpha} \overline b_{\beta}
\overline \varphi$ since $\alpha\ne \beta$. \
We put
$$
\varphi_-^\prime := P_{\mathcal H(b_{\alpha}
b_{\beta})}(\varphi_-)\quad\hbox{and}\quad \varphi_-^{\prime\prime}
:= P_{b_{\alpha} b_{\beta}H^2}(\varphi_-).
$$
We then have
\begin{equation}\label{6.8-3}
\psi_+=-\overline b_{\alpha} \overline b_{\beta}
\varphi_-^{\prime\prime}\quad\hbox{and}\quad
\psi_-=-{{b_{\alpha}b_{\beta}}}(\varphi_+
+\overline{\varphi_-^\prime}).
\end{equation}
It thus follows from (\ref{6.8-3}) that
\begin{equation}\label{6.8-4}
\theta_1 \overline{b}=\psi_-=-{{b_{\alpha}b_{\beta}}}(\varphi_+
+\overline{\varphi_-^\prime}), \quad\hbox{so that}\ \
   \overline{b}= - b_{\alpha} b_{\beta} (\theta_3\overline{d}
           +\overline{\theta_1}\overline{\varphi_-^\prime})\in \overline{H^2}
\end{equation}
which gives
$$
\theta_3\overline{d} +
\overline{\theta_1}\overline{\varphi_-^\prime}\in
\overline{H^2},\quad \hbox{and hence,}\ \ d\in \theta_3 H^2,
$$
which implies that $\theta_3$ is a constant because $\theta_3$ and
$d$ are coprime. \
We therefore have $\varphi_+=\theta_1
\overline{d}$. \
It thus follows from (\ref{6.7-5}) together with
again Lemma \ref{lem3.1} that
$$
\theta_1=b_{\alpha}\theta_1'\quad\hbox{(some inner function
$\theta_1^\prime$)}.
$$
But since by (\ref{6.8-4}),
$$
\overline{b}= - b_{\alpha} b_{\beta}\overline{(d+\theta_1
\varphi_-^\prime)}\in\overline{H^2},
$$
so that
$$
d+\theta_1 \varphi_-^\prime \in b_\alpha b_\beta H^2,
$$
which implies that $d(\alpha)=0$, a contradiction because $\theta_1$
and $d$ are coprime. \
This proves (\ref{6.8-2}).

\medskip

We now write
$$
\Phi\equiv \left[\begin{matrix} \overline b_\alpha&\varphi\\
\psi&\overline b_\alpha \end{matrix}\right] \equiv \Phi_-^*+\Phi_+
=\begin{bmatrix} b_\alpha & \psi_-\\
\varphi_-& b_\alpha \end{bmatrix}^{\ast}+\begin{bmatrix} 0&\varphi_+\\
\psi_+& 0\end{bmatrix}\,,
$$
where
$$
\varphi_-:=\theta_0\overline{a}\quad\hbox{and}\quad \psi_-:=\theta_1
\overline{b}\qquad (a\in \mathcal H_{\theta_0},\  b \in \mathcal
H_{\theta_1}).
$$
Moreover, we have
\begin{equation}\label{6.7-77}
\begin{cases}
\overline b_\alpha-k_2 \overline{\varphi_+}\in H^2,\quad
\overline{\theta}_1 b-k_4\overline{\varphi_+}\in H^2\\
\overline b_\alpha-k_3 \overline{\psi_+}\in H^2,\quad
\overline{\theta}_0 a-k_1\overline{\psi_+}\in H^2
\end{cases}
\end{equation}
and the following Toeplitz operators are all hyponormal: \
\begin{equation}\label{6.7-55}
T_{\overline b_\alpha+\varphi_+},\ \ T_{\overline{\theta}_1 b +
\varphi_+},\ \ T_{\overline b_\alpha+\psi_+},\ \
T_{\overline{\theta}_0 a + \psi_+}.
\end{equation}
Note that $\varphi_+\psi_+$ is not identically zero, so that
$\hbox{det}\,\Phi_+$ is not. \
Put
$$
\theta_0=b_{\alpha}^m\theta_0^{\prime}\quad\hbox{and}\quad
\theta_1=b_{\alpha}^n \theta_1^{\prime}\qquad (m,n\ge 0;\
\theta_0^{\prime}(\alpha)\ne 0,\ \theta_1^{\prime}(\alpha)\ne 0).
$$
We now claim that
\begin{equation} \label{6.7}
m=n=0\quad\hbox{or}\quad m=n=1.
\end{equation}
We split the proof of (\ref{6.7}) into three cases. \

\bigskip

%
%
%

\noindent  {\bf Case 1} ($m\ne 0$ and $n=0$): \quad In this case, we
have $a(\alpha)\ne 0$ because $\theta_0(\alpha)=0$ and $\theta_0$
and $a$ are coprime. \
We first claim that
\begin{equation}\label{6.7-p}
m=1.
\end{equation}
To show this we assume to the contrary that $m \geq 2$. \
Write
$$
\gamma:=-\frac{a(\alpha)}{\theta_1(\alpha)} \quad
\hbox{and} \quad  \nu:=\frac{1}{\sqrt{|\gamma|^2+1}}.
$$
To get the left coprime factorization of $\Phi_-$, applying Lemma
\ref{lem4.1}(b) for $\widetilde{\Phi_-}$ gives \
\medskip
$$
\widetilde{\Phi_-}
=\begin{bmatrix} \widetilde{b}_{\alpha}&\widetilde{\theta_0} \overline{\widetilde{a}}\\
\widetilde{\theta_1}\overline{\widetilde{b}}&\widetilde{b}_{\alpha}
\end{bmatrix} =\widetilde\Omega_2\widetilde
B^*\quad\hbox{(right coprime factorization)}\,,
$$
where
$$
\Omega_2:=
\nu \begin{bmatrix} b_{\alpha}^{m-1}\theta_0'&\gamma \theta_1\\
-\overline{\gamma} \theta_0 &b_{\alpha}\theta_1
\end{bmatrix}\quad\hbox{and}\quad B\in H^2_{M_n}\,,
$$
which gives
$$
\Phi_- =\begin{bmatrix} b_{\alpha}&\theta_1 \overline{b}\\
\theta_0\overline{a}&b_{\alpha} \end{bmatrix} = B^*\Omega_2
\quad\hbox{(left coprime factorization)}\,.
$$
On the other hand, by a scalar-valued version of Lemma \ref{lem3.1}
and (\ref{6.7-55}), we can see that \
$$
\varphi_+=b_{\alpha} \theta_1 \theta_3\overline{d}\ \ \hbox{and}\ \
\psi_+=\theta_0\theta_2\overline{c}\ \ \hbox{for some inner
functions $\theta_2,\theta_3$,}
$$
where $d\in\mathcal{H}_{z b_{\alpha} \theta_1 \theta_3}$ and $c\in
\mathcal H_{z\theta_0\theta_2}$. \
Thus  in particular, $c(\alpha)\ne
0$ and $d(\alpha)\ne 0$. \
We first observe that
\begin{equation}\label{6.7-6}
k_3(\alpha)=0\quad\hbox{and}\quad k_4(\alpha)=0:
\end{equation}
indeed, in (\ref{6.7-77}),
$$
\begin{aligned}
\overline b_{\alpha}-k_3 \overline{\psi_+}\in
H^2&\Longrightarrow\overline b_{\alpha}-k_3
\overline{\theta_0\theta_2}c\in
H^2\\
&\Longrightarrow b_{\alpha}^{m-1}\theta_0'\theta_2 -k_3c\in b_{\alpha}^{m}\theta_0'\theta_2H^2\\
&\Longrightarrow k_3(\alpha)=0\quad\hbox{(since $m\ge 2$)}
\end{aligned}
$$
and
$$
\begin{aligned}
\overline{\theta}_1 b-k_4\overline{\varphi_+}\in
H^2&\Longrightarrow\overline{\theta}_1 b-k_4\overline b_{\alpha}\overline{\theta_1 \theta_3}d\in H^2\\
&\Longrightarrow b_{\alpha} \theta_3 b-k_4d\in b_{\alpha} \theta_1
\theta_3H^2\\
&\Longrightarrow k_4(\alpha)=0\,,
\end{aligned}
$$
which proves (\ref{6.7-6}). \
Write
$$
\theta_2=b_{\alpha}^q \theta_2'\quad\hbox{and}\quad
\theta_3=b_{\alpha}^p \theta_3'\quad (\theta_2'(\alpha)\neq 0,\
\theta_3'(\alpha) \neq 0).
$$
Then we can write
$$
\Phi_+
=\begin{bmatrix} 0& b_{\alpha}\theta_1\theta_3 \overline d\\
     \theta_0\theta_2\overline c &0 \end{bmatrix}
=\begin{bmatrix} 0&b_{\alpha}^{p+1} \theta_1
\theta_3'\overline{d}\\
b_{\alpha}^{m+q}\theta_0'\theta_2'\overline{c}& 0\end{bmatrix}.
$$
We suppose that $p+1 < m+q$ and write $r:= (m+q)-(p+1)>0$. \
Then
$$
\Phi_+=(b_{\alpha}^{m+q}\theta_1\theta_3^\prime\theta_0^\prime\theta_2^\prime) I_2
\begin{bmatrix} 0&\theta_1\theta_3^\prime c\\b_{\alpha}^r \theta_0^\prime\theta_2^\prime d&0\end{bmatrix}^*
\equiv (\theta I_2)A^*,
$$
where
$\theta:=b_{\alpha}^{m+q}\theta_1\theta_3^\prime\theta_0^\prime\theta_2^\prime$. \
Observe that
$$
 A\Omega_2^*=\nu \begin{bmatrix}0& \theta_1\theta_3^\prime c\\ b_{\alpha}^r \theta_0^\prime\theta_2^\prime d&0
\end{bmatrix}
\begin{bmatrix} b_{\alpha}^{m-1}\theta_0'&\gamma \theta_1\\
-\overline{\gamma} \theta_0 &b_{\alpha}\theta_1
\end{bmatrix}^*
=\nu \begin{bmatrix}\overline{\gamma}\theta_3^{\prime}c & \overline b_{\alpha}\theta_3^{\prime}c\\
b_{\alpha}^{r-m+1} \theta_2^{\prime}d &-\gamma
b_{\alpha}^{r-m}\theta_2^{\prime}d
\end{bmatrix}\,.
$$
If $r\ge m-1$, then we have
$$
H_{A\Omega_2^*}\begin{bmatrix}0\\
1\end{bmatrix}
 =\nu \begin{bmatrix}
H_{\overline{b}_{\alpha}}(\theta_3^{\prime}c)\\
-\gamma H_{b_{\alpha}^{r-m}}(\theta_2^{\prime}d)
\end{bmatrix}.
$$
Put
$$
\delta:=\frac{\sqrt{1-|\alpha|^2}}{1-\overline{\alpha}z} \quad
\hbox{and} \quad
\delta_1:= \widetilde{\delta}=\frac{\sqrt{1-|\alpha|^2}}{1-{\alpha}z}
$$
and observe that $\hbox{ran}\, H_{\overline{b}_{\alpha}}=\mathcal
H_{b_{\overline{\alpha}}}=\bigvee \{\delta_1\}$. \
Since
$(\theta_3^{\prime}c)(\alpha)\neq 0$ and $(\theta_2'd)(\alpha) \neq
0$, it follows from (\ref{3.15}) that
\begin{equation}\label{6.7-81}
\begin{bmatrix} \delta_1\\ \beta \delta_1  \end{bmatrix} \in\hbox{cl ran}\, H_{A\Omega_2^*}\subseteq
\hbox{ker}\,(I-T_{\widetilde{K}}T_{\widetilde{K}}^*),
\end{equation}
where $\beta\in\mathbb C$ is possibly zero (when $r\ge m$). \
We observe
that if $k\in H^2$, then since $\frac{1}{1-\alpha z}$ is the
reproducing kernel for $\overline\alpha$, we can get
\begin{equation}\label{6.7-80}
T_{k(\overline{z})}\delta_1=k(\alpha)\delta_1:
\end{equation}
indeed, if $k\in H^2$ and $n\geq 0$, then
$$
\langle k(\overline{z})\delta_1, \ z^n \rangle=\langle \delta_1, \
\overline{k(\overline{z})}z^n \rangle = \overline{\langle
\widetilde{k}z^n, \ \delta_1
\rangle}=\sqrt{1-|\alpha|^2}\overline{\widetilde{k(\overline{\alpha})}
      \overline{\alpha}^n}=\sqrt{1-|\alpha|^2}k(\alpha)\alpha^n\,,
$$
so that
$$
T_{k(\overline{z})}\delta_1=P(k(\overline{z})\delta_1)
=\sqrt{1-|\alpha|^2}k(\alpha)\sum_{n=0}^{\infty}\alpha^nz^n=k(\alpha)\frac{\sqrt{1-|\alpha|^2}}{1-\alpha
z}=k(\alpha)\delta_1\,,
$$
which proves (\ref{6.7-80}). \
It thus follows from (\ref{6.7-6}), (\ref{6.7-81}) and (\ref{6.7-80}) that
$$
\aligned
\begin{bmatrix} \delta_1\\ \beta \delta_1  \end{bmatrix}=T_{\widetilde{K}}T_{\widetilde{K}}^*\begin{bmatrix} \delta_1\\ \beta \delta_1
\end{bmatrix}
& = \begin{bmatrix} T_{\widetilde k_1} & T_{\widetilde k_3}\\
      T_{\widetilde k_2} & T_{\widetilde k_4}\end{bmatrix}
           \begin{bmatrix} T_{\overline{\widetilde k}_1} & T_{\overline{\widetilde k}_2}\\
              T_{\overline{\widetilde k}_3} & T_{\overline{\widetilde k}_4}\end{bmatrix}
                    \begin{bmatrix} \delta_1\\ \beta \delta_1  \end{bmatrix}\\
& = \begin{bmatrix}T_{\widetilde{k_1}}& T_{\widetilde{k_3}}\\
   T_{\widetilde{k_2}}&T_{\widetilde{k_4}}\end{bmatrix}
     \begin{bmatrix}T_{k_1(\overline{z})}\delta_1+\beta T_{k_2(\overline{z})}\delta_1\\
       T_{k_3(\overline{z})}\delta_1+\beta T_{k_4(\overline{z})}\delta_1\end{bmatrix}\\
& = \begin{bmatrix}T_{\widetilde{k_1}}& T_{\widetilde{k_3}}\\
T_{\widetilde{k_2}}&T_{\widetilde{k_4}}\end{bmatrix}
\begin{bmatrix}(k_1(\alpha)+\beta k_2(\alpha))\delta_1\\0\end{bmatrix}\\
& = \begin{bmatrix}\widetilde{k_1}(k_1(\alpha)+\beta k_2(\alpha))\delta_1\\ \widetilde{k_2}(k_1(\alpha)+\beta
k_2(\alpha))\delta_1\end{bmatrix},
\endaligned
$$
which implies that $k_1$ and $k_2$ are nonzero constants. \
Thus by (\ref{6.7-77}),
$$
\overline b_{\alpha} -k_2\overline{\varphi_+}\in H^2 \Longrightarrow
b_{\alpha}\overline{\varphi}_+ \in H^2 \Longrightarrow
\overline{\theta_1\theta_3}d \in H^2\,,
$$
which implies that $\theta_1\theta_3$ is a constant. \
Without loss of generality we may assume $\theta_1\theta_3=1$ and  $\psi_-=0$. \
Similarly, from (\ref{6.7-77}), $\overline{\theta_0}a-k_1
\overline{\psi_+}\in H^2$, i.e., $\overline{\theta_0}a-k_1\overline{\theta_0}\overline{\theta_2}c\in H^2$
implies $\theta_2=1$. \
But since by (\ref{6.5}), $|\varphi|=|\psi|$, we have \
$$
|b_{\alpha}\overline d +\overline{\theta_0} a|=|\varphi_+
+\overline{\varphi_-}|=|\psi_+|=|\theta_0 \overline c|
\quad\hbox{(where $a\in\mathcal{H}_{\theta_0},\ c\in\mathcal H_{z\theta_0}$ )}\,,
$$
which implies
$$
{b_{\alpha}\theta_0} (b_{\alpha}\overline d +\overline{\theta_0}
a)(\overline{b}_{\alpha} d +{\theta_0} \overline a) =
{b_{\alpha}\theta_0} c \overline{c}\,,
$$
so that
\begin{equation}\label{55}
ad= b_{\alpha}\Bigl(
({\theta_0}\overline c){c}-({\theta_0}\overline d){d}
  -({\theta_0} \overline a) ({\theta_0} \overline d) b_{\alpha}
    -({\theta_0}\overline a) {a} \Bigr)\,.
\end{equation}
But since $a, c\in\mathcal{H}_{z\theta_0}$, and $d\in\mathcal H_{z
b_{\alpha}}$, it follows that ${\theta}_0 \overline a,
{\theta}_0 \overline c$ and ${\theta}_0 \overline d$ are in
${H^2}$. \
Thus (\ref{55}) implies that $(ad)(\alpha)= 0$,
a contradiction. \
Therefore this case cannot occur.

\medskip


If instead $r< m-1$ then the same argument as before leads to a contradiction.
Therefore this case cannot occur. \
Moreover, by the same argument as in the case $p+1<m+q$, the case $p+1\ge m+q$
cannot also occur. \
This proves (\ref{6.7-p}). \


\bigskip

Now by Lemma \ref{lem4.1}(b), we
can write \
$$
\Phi_-=\begin{bmatrix} \theta_0^\prime & a \\ b_\alpha b &
\theta_1\end{bmatrix}^*
           \begin{bmatrix} \theta_0 &0\\ 0& b_\alpha \theta_1\end{bmatrix} \equiv
                B^*\Omega_2\ \ \hbox{(left coprime factorization).}
$$
Observe
$$
\Omega_2\equiv \begin{bmatrix} \theta_0 &0\\
0& b_\alpha \theta_1\end{bmatrix}= \begin{bmatrix} b_\alpha &0\\ 0& b_\alpha\end{bmatrix}
\begin{bmatrix} \theta_0' &0\\ 0&\theta_1\end{bmatrix}.
$$
Since $\theta_0'(\alpha)\ne 0$ and $\theta_1(\alpha)\ne 0$, it follows
from Lemma \ref{lem5.5} that
$\begin{bmatrix} b_\alpha&0\\ 0&b_\alpha\end{bmatrix}$ and
$\begin{bmatrix} \theta_0' &0\\ 0&\theta_1\end{bmatrix}$ are coprime,
so that by Corollary \ref{cor3.8}, $T_\Phi$ should be normal. \
Since $\hbox{det}\,\Phi_+$ is not identically zero, it
follows form Lemma \ref{lem1.3} that $\Phi_+-\Phi_-U\in M_n$ for some
constant unitary matrix $U\equiv \left[\begin{smallmatrix} c_1&c_2\\
c_3&c_4\end{smallmatrix}\right]$. \
We observe
$$
\begin{aligned}
\Phi_+-\Phi_-U\in M_n
      &\Longleftrightarrow \begin{bmatrix} 0& b_\alpha \theta_1\theta_3\overline d\\
                            \theta_0\theta_2\overline c&0\end{bmatrix}
                              - \begin{bmatrix} b_\alpha &\theta_1\overline b\\ \theta_0\overline a& b_\alpha\end{bmatrix}
                                  \begin{bmatrix} c_1&c_2\\ c_3&c_4\end{bmatrix}
                                  =\begin{bmatrix}\xi_1&\xi_2\\ \xi_3&\xi_4\end{bmatrix}\quad (\xi_i\in\mathbb C)\\
      &\Longrightarrow
         \begin{cases} c_1 b_\alpha+c_3\theta_1\overline b=-\xi_1\\
                       c_4 b_\alpha+c_2\theta_0\overline a=-\xi_4
         \end{cases}\\
      &\Longrightarrow  c_1=0,\ c_4\ne 0\\
      &\Longrightarrow U=\begin{bmatrix} 0&c_2\\ c_3&c_4\end{bmatrix}\quad(c_4\ne 0),
\end{aligned}
$$
which contradicts the fact that $U$ is unitary. \
Thus this case cannot occur.

\bigskip

%
%

\noindent {\bf Case 2} ($m=0$ and $n\ne 0$): \quad This case is
symmetrical to Case 1. \
The proof is identical to that of Case 1. \
Therefore this case cannot occur either. \

\bigskip

%
%

\noindent  {\bf Case 3} ($m\ne 0$, $n\ne 0$ and $m\ge 2$ or $n\ge
2$): \quad In this case, $a(\alpha)\ne 0$ and $b(\alpha)\ne 0$ since
$\theta_0$ and $a$ are coprime and $\theta_1$ and $b$ are coprime. \
By Lemma \ref{lem4.1}(d), we can write \
$$
\Phi_-=\begin{bmatrix} b_\alpha^{m-1}\theta_0^\prime & a \\ b &
b_\alpha^{n-1} \theta_1^\prime\end{bmatrix}^*
           \begin{bmatrix} \theta_0 &0\\ 0&\theta_1\end{bmatrix} \equiv
                B^*\Omega_2\ \ \hbox{(left coprime factorization).}
$$
We first suppose $m\ge 2$. \
By a scalar-valued version of Lemma \ref{lem3.1} and (\ref{6.7-55}), we
can see that \
$$
\varphi_+= \theta_1 \theta_3\overline{d}\ \ \hbox{and}\ \
\psi_+=\theta_0\theta_2\overline{c}\ \ \hbox{for some inner
functions $\theta_2,\theta_3$,}
$$
where $d\in\mathcal{H}_{z \theta_1 \theta_3}$ and $c\in \mathcal H_{z\theta_0\theta_2}$. \
Note that $c(\alpha)\ne 0$ and $d(\alpha)\ne 0$. \
We first observe
\begin{equation}\label{55-55}
k_3(\alpha)=0:
\end{equation}
indeed,
in (\ref{6.7-77}),
$\overline b_\alpha -k_3 \overline{\psi_+}\in H^2$ implies
$b_\alpha^{m-1}\theta_0'\theta_2 -k_3 c\in b_\alpha^{m}\theta_0'\theta_2H^2,$
giving $k_3(\alpha)=0$. \
Write
$$
\theta_2= b_\alpha^q \theta_2'\quad\hbox{and}\quad \theta_3=b_\alpha^p
\theta_3'\quad \hbox{($\theta_2'(\alpha)\neq 0$, $\theta_3'(\alpha) \neq 0$)}.
$$
Then we can write
$$
\Phi_+
=\begin{bmatrix} 0& \theta_1\theta_3 \overline d\\
     \theta_0\theta_2\overline c &0 \end{bmatrix}
=\begin{bmatrix} 0&b_\alpha^{n+p}\, \theta_1'
\theta_3'\overline{d}\\
b_\alpha^{m+q}\,\theta_0'\theta_2'\overline{c}& 0\end{bmatrix}.
$$
If $n+p \le m+q$, write $r:= (m+q)-(n+p)\ge 0$. \
Then
$$
\Phi_+=
(b_\alpha^{m+q}\theta_1'\theta_3^\prime\theta_0^\prime\theta_2^\prime)I_2
\begin{bmatrix} 0&\theta_1'\theta_3^\prime c\\b_\alpha^r \theta_0^\prime\theta_2^\prime d&0\end{bmatrix}^*
\equiv (\theta I_2) A^*,
$$
where
$\theta:=b_\alpha^{m+q}\theta_1'\theta_3^\prime\theta_0^\prime\theta_2^\prime$. \
Observe that
$$
A\Omega_2^*=
\begin{bmatrix} 0& \theta_1'\theta_3^\prime c\\ b_\alpha^r \theta_0^\prime\theta_2^\prime d&0\end{bmatrix}
\begin{bmatrix} \theta_0& 0\\ 0 & \theta_1\end{bmatrix}^*
= \begin{bmatrix} 0 & \overline b_\alpha^n \theta_3^{\prime}c\\
b_\alpha^{r-m} \theta_2^{\prime}d & 0\end{bmatrix}\,,
$$
so that
$$
H_{A\Omega_2^*} \begin{bmatrix} 0\\ b_\alpha^{n-1}\end{bmatrix}
  = \begin{bmatrix} H_{\overline b_\alpha} (\theta_3' c)\\ 0\end{bmatrix}\,.
$$
Since $(\theta_3'c)(\alpha)\ne 0$, it follows from
(\ref{3.15}) and the same argument as (\ref{6.7-81}) that
$$
\begin{bmatrix} \delta_1\\ 0\end{bmatrix} \in\hbox{cl ran}\, H_{A\Omega_2^*}\subseteq
\hbox{ker}\,(I-T_{\widetilde{K}}T_{\widetilde{K}}^*)\,.
$$
Since by (\ref{55-55}), $k_3(\alpha)=0$, it follows  from (\ref{6.7-80}) that
$$
\aligned
\begin{bmatrix} \delta_1\\ 0 \end{bmatrix}
       = T_{\widetilde{K}}T_{\widetilde{K}}^*\begin{bmatrix} \delta_1\\ 0 \end{bmatrix}
   &  =\begin{bmatrix}T_{\widetilde{k_1}}& T_{\widetilde{k_3}}\\
                      T_{\widetilde{k_2}}&T_{\widetilde{k_4}}\end{bmatrix}
         \begin{bmatrix} T_{\overline{\widetilde k}_1} & T_{\overline{\widetilde k}_2}\\
              T_{\overline{\widetilde k}_3} & T_{\overline{\widetilde k}_4}\end{bmatrix}
                  \begin{bmatrix} \delta_1\\ 0 \end{bmatrix} \\
   & =\begin{bmatrix}T_{\widetilde{k_1}}& T_{\widetilde{k_3}}\\
             T_{\widetilde{k_2}}&T_{\widetilde{k_4}}\end{bmatrix}
                \begin{bmatrix}k_1(\alpha)\delta_1\\k_3(\alpha)\delta_1\end{bmatrix}\\
   & =         \begin{bmatrix} k_1(\alpha)\widetilde{k_1}\delta_1\\
              k_1(\alpha)\widetilde{k_2}\delta_1\end{bmatrix},
\endaligned
$$
which implies that $k_2=0$. \
This leads to a contradiction with (\ref{6.7-77}). \
If $n+p > m+q$, then a similar argument leads to a contradiction. \
Also if instead $n\ge 2$, then the argument is symmetrical with the case $m\ge 2$. \
Thus this case cannot occur. \

\medskip

Consequently, all three cases cannot occur. \
This proves (\ref{6.7}). \

\bigskip

%
%

Now it suffices to consider the case $m=n=0$ and the case $m=n=1$. \

\medskip

\noindent  {\bf Case A} ($m=n=0$)\quad In this case, by Lemma
\ref{lem4.1}(c) we can write
$$
\Phi_-=\begin{bmatrix} \theta_0 & b_\alpha a \\ b_\alpha b &
\theta_1\end{bmatrix}^*
           \begin{bmatrix} b_\alpha \theta_0 &0\\ 0& b_\alpha\theta_1\end{bmatrix} \equiv
                B^*\Omega_2\ \ \hbox{(left coprime factorization).}
$$
Observe that
$$
\Omega_2\equiv \begin{bmatrix} b_\alpha \theta_0 &0\\
0& b_\alpha \theta_1\end{bmatrix}= \begin{bmatrix} b_\alpha &0\\ 0& b_\alpha \end{bmatrix}
\begin{bmatrix} \theta_0 &0\\ 0&\theta_1\end{bmatrix}.
$$
Since $\theta_0(\alpha)\ne 0$ and $\theta_1(\alpha)\ne 0$, it follows
from Lemma \ref{lem5.5} that
$\begin{bmatrix} b_\alpha & 0\\ 0 & b_\alpha \end{bmatrix}$ and
$\begin{bmatrix} \theta_0 &0\\ 0&\theta_1\end{bmatrix}$ are coprime,
so that by Corollary \ref{cor3.8}, $T_\Phi$ should be normal. \
Since $\hbox{det}\,\Phi_+$ is not identically zero, it
follows form Lemma \ref{lem1.3} that $\Phi_+-\Phi_-U\in M_n$ for some
constant unitary matrix $U\equiv \left[\begin{smallmatrix} c_1&c_2\\
c_3&c_4\end{smallmatrix}\right]$. \
We observe
\begin{equation}\label{6.7.81}
\begin{aligned}
\Phi_+-\Phi_-U\in M_n
      &\Longleftrightarrow \begin{bmatrix} 0& \theta_1\theta_3\overline d\\
                            \theta_0\theta_2\overline c&0\end{bmatrix}
                              - \begin{bmatrix} b_\alpha &\theta_1\overline b\\ \theta_0\overline a& b_\alpha\end{bmatrix}
                                  \begin{bmatrix} c_1&c_2\\ c_3&c_4\end{bmatrix}
                                     =\begin{bmatrix}\xi_1&\xi_2\\ \xi_3&\xi_4\end{bmatrix}\quad (\xi_i\in\mathbb C)\\
      &\Longleftrightarrow
         \begin{cases} c_1 b_\alpha+c_3\theta_1\overline b=-\xi_1\\
                       c_4 b_\alpha+c_2\theta_0\overline a=-\xi_4\\
                       \theta_1\theta_3\overline d=c_2 b_\alpha + c_4\theta_1\overline b+\xi_2\\
                       \theta_0\theta_2\overline c=c_3 b_\alpha + c_1\theta_0\overline a+\xi_3
         \end{cases}\\
      &\Longrightarrow
        \begin{cases} c_1=0,\ \theta_1\overline b=\hbox{a constant}\\
                      c_4=0,\ \theta_0 \overline a=\hbox{a constant}\\
                      \varphi_+=\theta_1\theta_3 \overline d =c_2 b_\alpha+\hbox{a constant}\\
                      \psi_+=\theta_0\theta_2 \overline c= c_3 b_\alpha+\hbox{a constant}\,.
         \end{cases}
\end{aligned}
\end{equation}
Since $U$ is unitary we have $c_2=e^{i\omega_1}$ and
$c_3=e^{i\omega_2}$ ($\omega_1,\omega_2\in [0,2\pi)$). \
Thus we have
$$
\varphi=e^{i\omega_1} b_\alpha + \beta_1\quad\hbox{and}\quad
\psi=e^{i\omega_2} b_\alpha +\beta_2\quad(\beta_1,\beta_2\in\mathbb C).
$$
But since $|\varphi|=|\psi|$, it follows that
$$
\varphi=e^{i\omega} b_\alpha + \zeta\quad\hbox{and}\quad
\psi=e^{i\delta}\varphi
          \quad\hbox{($\omega,\delta\in [0,2\pi)$, $\zeta\in\mathbb C$))}.
$$

\bigskip

%
%

\noindent {\bf Case B ($m=n=1$)}:\quad We split the proof into two
subcases.

\bigskip

\noindent {\bf Case B-1 ($m=n=1$; $(ab)(\alpha)\ne
(\theta_0^\prime\theta_1^\prime)(\alpha)$)}:\quad In this case, by
Lemma \ref{lem4.1}(d), we can write
$$
\Phi_-=\begin{bmatrix} \theta_0^\prime & a \\ b &
\theta_1^\prime\end{bmatrix}^*
           \begin{bmatrix} b_\alpha \theta_0^\prime &0\\ 0& b_\alpha\theta_1^\prime \end{bmatrix} \equiv
                B^*\Omega_2\ \ \hbox{(left coprime factorization).}
$$
Observe
$$
\Omega_2\equiv \begin{bmatrix} b_\alpha\theta_0' &0\\
0&b_\alpha\theta_1'\end{bmatrix}= \begin{bmatrix} b_\alpha&0\\ 0&b_\alpha\end{bmatrix}
\begin{bmatrix} \theta_0' &0\\ 0&\theta_1'\end{bmatrix}.
$$
Since $\theta_0'(\alpha)\ne 0$ and $\theta_1'(\alpha)\ne 0$,
it follows from Lemma \ref{lem5.5} that
$\begin{bmatrix} b_\alpha&0\\ 0&b_\alpha\end{bmatrix}$ and
$\begin{bmatrix} \theta_0' &0\\ 0&\theta_1'\end{bmatrix}$ are coprime,
so that by Corollary \ref{cor3.8}, $T_\Phi$ should be normal. \
By the same argument as in (\ref{6.7.81}), we can see that
$$
\theta_2=\theta_3=1 \quad\hbox{and}\quad
\theta_0^\prime=\theta_1^\prime =1
$$
and we can write
$$
\Phi_+=\begin{bmatrix}0&\varphi_+\\ \psi_+&0
\end{bmatrix}\quad\hbox{and}\quad
\Phi_-^*=\begin{bmatrix}\overline{b}_\alpha & a\, \overline b_\alpha\\
b \, \overline b_\alpha & \overline b_\alpha
\end{bmatrix}\qquad (a,b\in\mathbb C; \ a\ne 0,\ b\ne 0).
$$
Since $T_\Phi$ is normal we have
$$
\begin{bmatrix}
H_{\overline{\varphi_+}}^*H_{\overline{\varphi_+}}&0\\0&H_{\overline{\psi_+}}^*H_{\overline{\psi_+}}
\end{bmatrix}=\begin{bmatrix}
(1+|b|^2)H_{\overline b_\alpha}^* H_{\overline b_\alpha}&(a + \overline{b})H_{\overline b_\alpha}^* H_{\overline b_\alpha}\\
(\overline{a}+b)H_{\overline b_\alpha}^* H_{\overline b_\alpha} & (1+|a|^2)H_{\overline b_\alpha}^* H_{\overline
b_\alpha}\end{bmatrix},
$$
which implies that
\begin{equation}\label{6.17}
\begin{cases}
b=-\overline{a}\\
H_{\overline{\varphi_+}}^*H_{\overline{\varphi_+}}=(1+|b|^2)H_{\overline b_\alpha}^* H_{\overline b_\alpha}\\
H_{\overline{\psi_+}}^*H_{\overline{\psi_+}}=(1+|a|^2)H_{\overline b_\alpha}^* H_{\overline b_\alpha}.
\end{cases}
\end{equation}
Since $ab\ne (\theta_0'\theta_1')(\alpha)$, we have $1\ne |ab|=|a|^2$, i.e.,
$|a|\ne 1$. \
We thus have
$$
\varphi_+=e^{i\theta_1}\sqrt{1+|a|^2}\,b_\alpha  +
\beta_1\quad\hbox{and}\quad
\psi_+=e^{i\theta_2}\sqrt{1+|a|^2}\,b_\alpha + \beta_2,
$$
($\beta_1,\beta_2\in\mathbb C;\ \theta_1,\theta_2\in [0,2\pi)$)
which implies that
$$
\varphi=a\, \overline b_\alpha + e^{i
\theta_1}\sqrt{1+|a|^2}\,b_\alpha + \beta_1\quad\hbox{and}\quad
\psi=-\overline{a}\, \overline b_\alpha + e^{i
\theta_2}\sqrt{1+|a|^2}\,b_\alpha+\beta_2.
$$
Since $|\varphi|=|\psi|$, a straightforward calculation shows that
\begin{equation}\label{6.18}
\varphi=\mu\, \overline b_\alpha + e^{i
\theta}\sqrt{1+|\mu|^2}\,b_\alpha + \zeta \quad\hbox{and}\quad
\psi=e^{i\,(\pi-2\,{\rm arg}\,\mu)}\varphi,
\end{equation}
where $\mu\ne 0,\ |\mu|\ne 1,\ \zeta\in\mathbb C$, and $\theta\in
[0,2\pi)$. \

\bigskip

%
%

\noindent {\bf Case B-2} ($m=n=1$; $(ab)(\alpha)=(\theta_0^{\prime}
\theta_1^{\prime})(\alpha)$):\quad In this case,
$\theta_i^{\prime}(\alpha)\ne 0$  for each $i=0,1$. \
By  a scalar-valued
version of Lemma \ref{lem3.1} and (\ref{6.7-5}), we can see that
$$
\varphi_+=\theta_1 \theta_3\overline{d}\ \ \hbox{and}\ \
\psi_+=\theta_0\theta_2\overline{c}\ \ \hbox{for some inner
functions $\theta_2,\theta_3$,}
$$
where $d\in\mathcal{H}_{z\theta_1 \theta_3}$ and $c\in \mathcal
H_{z\theta_0\theta_2}$. \
Thus  in particular, $c(\alpha)\ne 0$ and $d(\alpha)\ne 0$. \
Let
$$
\theta_2=b_{\alpha}^q \theta_2'\quad\hbox{and}\quad
\theta_3=b_{\alpha}^p \theta_3' \quad\hbox{($\theta_2'(\alpha)\neq
0,\ \theta_3'(\alpha) \neq 0$)}.
$$
To get the left coprime factorization of $\Phi_-$, applying Lemma
\ref{lem4.1}(d) for $\widetilde\Phi_-$ gives
$$
\widetilde\Phi_- =
\begin{bmatrix}
\widetilde b_\alpha & \widetilde\theta_0 \overline{\widetilde a}\\
\widetilde\theta_1 \overline{\widetilde b} & \widetilde b_\alpha
\end{bmatrix}
=\widetilde \Omega_2 \widetilde B^*\quad \hbox{(right coprime
factorization)},
$$
where
$$
\Omega_2:= \nu \begin{bmatrix}
            \theta_0 & -\overline \gamma\theta_1 \\
               {\gamma}\theta_0^\prime & \theta_1^\prime \end{bmatrix}
                \quad  ({\gamma}=-\frac{b(\alpha)}{\theta_0^{\prime}(\alpha)}
                    =-\frac{\theta_1^{\prime}(\alpha)}{a(\alpha)} )
$$
Then we get
\begin{equation}\label{6.8-90}
\Phi_-=
\begin{bmatrix}
b_\alpha &\theta_1 \overline b\\ \theta_0\overline a &b_\alpha \end{bmatrix}
=B^*\Omega_2\quad\hbox{(left coprime factorization)}\,.
\end{equation}
We now claim that
\begin{equation}\label{6.8-88}
p=q.
\end{equation}
We first assume that $p<q$. Then $\theta_2(\alpha)=0$. \
Thus by (\ref{6.7-77}), we have $k_1(\alpha)=k_3(\alpha)=0$.  \
Write
$s:=q-p\ge 1$. \
In this case we can write
$$
\Phi_+
=\begin{bmatrix} 0&\theta_1\theta_3 \overline d\\
\theta_0\theta_2\overline c&0\end{bmatrix}
=(b_{\alpha}^{q+1}\theta_1^\prime\theta_3^\prime\theta_0^\prime\theta_2^\prime) I_2
     \begin{bmatrix} 0&\theta_1^\prime\theta_3^\prime c\\
           b_\alpha^s \theta_0^\prime\theta_2^\prime d &0 \end{bmatrix}^*
\equiv (\theta I_2) A^*\,,
$$
where
$\theta:=b_{\alpha}^{q+1}\theta_1^\prime\theta_3^\prime\theta_0^\prime
\theta_2^\prime$. \
Observe that
\begin{equation}\label{6.8-89}
A\Omega_2^*
=\nu \begin{bmatrix} 0&\theta_1^\prime\theta_3^\prime c\\
           b_\alpha^s \theta_0^\prime\theta_2^\prime d &0 \end{bmatrix}
\begin{bmatrix}
            \theta_0 & -\overline \gamma\theta_1 \\
               {\gamma}\theta_0^\prime & \theta_1^\prime \end{bmatrix}^*
=\nu \begin{bmatrix}
          -{\gamma} \overline{b}_{\alpha}\theta_3^{\prime}c&\theta_3^{\prime}c\\
          b_{\alpha}^{s-1} \theta_2^{\prime}d & \overline\gamma b_{\alpha}^s \theta_2^{\prime}d\end{bmatrix}\,.
\end{equation}
Since $s\ge 1$, we have
$$
H_{A\Omega_2^*} \begin{bmatrix} 1\\0\end{bmatrix}
      =\nu \begin{bmatrix}-{\gamma} H_{\overline b_{\alpha}}(\theta_3^{\prime}c)\\
            0\end{bmatrix}.
$$
Since $(\theta_3^{\prime} c)(\alpha)\ne 0$, it follows from
(\ref{3.15}) that
$$
\begin{bmatrix} \delta_1\\0\end{bmatrix}\in\hbox{cl ran}\, H_{A\Omega_2^*}
\subseteq \hbox{ker}\,(I-T_{\widetilde K}T_{\widetilde K}^*).
$$
Thus we have
$$
\aligned
\begin{bmatrix} \delta_1\\ 0  \end{bmatrix}=T_{\widetilde{K}}T_{\widetilde{K}}^*\begin{bmatrix}
\delta_1\\0\end{bmatrix}
& =\begin{bmatrix}T_{\widetilde{k_1}}& T_{\widetilde{k_3}}\\
  T_{\widetilde{k_2}}&T_{\widetilde{k_4}}\end{bmatrix}
   \begin{bmatrix}T_{k_1(\overline{z})}\delta_1\\T_{k_3(\overline{z})}\delta_1\end{bmatrix}\\
& =\begin{bmatrix}T_{\widetilde{k_1}}& T_{\widetilde{k_3}}\\
   T_{\widetilde{k_2}}&T_{\widetilde{k_4}}\end{bmatrix}
     \begin{bmatrix}k_1(\alpha)\delta_1\\k_3(\alpha)\delta_1\end{bmatrix}\\
& = \begin{bmatrix}0\\0\end{bmatrix}\quad\hbox{(since $k_1(\alpha)=k_3(\alpha)=0$)}\,,
\endaligned
$$
which leads to a contradiction. \
If instead $p<q$ then the same argument leads to a contradiction. \
This proves (\ref{6.8-88}). \

Now since $p=q$, i.e., $s=0$, it follows again from
(\ref{6.8-89}) and (\ref{3.15}) that
\begin{equation}\label{6.7-9-0}
\begin{bmatrix} \delta_1\\ \beta \delta_1 \end{bmatrix}\in\hbox{cl ran}\, H_{A\Omega_2^*}
\subseteq \hbox{ker}\,(I-T_{\widetilde K}T_{\widetilde K}^*)\quad
(\beta \ne 0).
\end{equation}
We thus have
\begin{equation}\label{6.7-9}
\begin{aligned}
\begin{bmatrix} \delta_1\\ \beta \delta_1\end{bmatrix}
 & = \begin{bmatrix} T_{\widetilde k_1} & T_{\widetilde k_3}\\
      T_{\widetilde k_2} & T_{\widetilde k_4}\end{bmatrix}
           \begin{bmatrix} T_{\overline{\widetilde k}_1} & T_{\overline{\widetilde k}_2}\\
              T_{\overline{\widetilde k}_3} & T_{\overline{\widetilde k}_4}\end{bmatrix}
                  \begin{bmatrix} \delta_1\\ \beta \delta_1 \end{bmatrix} \\
 & =\begin{bmatrix} T_{\widetilde k_1} & T_{\widetilde k_3}\\
      T_{\widetilde k_2} & T_{\widetilde k_4}\end{bmatrix}
        \begin{bmatrix} (k_1(\alpha)+\beta k_2(\alpha))\delta_1\\ (k_3(\alpha)+\beta k_4(\alpha))\delta_1\end{bmatrix} \\
 & =\begin{bmatrix}
\Bigl(\widetilde{k}_1\bigl(k_1(\alpha)+\beta k_2(\alpha)\bigr)+\widetilde{k}_3\bigl(k_3(\alpha)+\beta k_4(\alpha)\bigr)\Bigr)\delta_1\\
\Bigl(\widetilde{k}_2\bigl(k_1(\alpha)+\beta
k_2(\alpha)\bigr)+\widetilde{k}_4\bigl(k_3(\alpha)+\beta
k_4(\alpha)\bigr)\Bigr)\delta_1
\end{bmatrix},
\end{aligned}
\end{equation}
which can be written as
\begin{equation}\label{6.7-10-10}
\alpha_1k_1+\alpha_2k_3=1 \quad \hbox{and} \quad
\alpha_1k_2+\alpha_2k_4=\overline{\beta}\,,
\end{equation}
where $\alpha_1=\overline{k_1(\alpha)+\beta k_2(\alpha)}$ and
$\alpha_2=\overline{k_3(\alpha)+\beta k_4(\alpha)}$.\
From (\ref{6.7-9-0}) we
also have
\begin{equation}\label{6.7-8-1}
     \left|\left|\begin{bmatrix} \delta_1\\ \beta\delta_1\end{bmatrix}\right|\right|_2
       = \left|\left|T_{\widetilde{K}}^*\begin{bmatrix} \delta_1\\ \beta\delta_1\end{bmatrix}\right|\right|_2
           =\left|\left| \begin{bmatrix} (k_1(\alpha)+\beta k_2(\alpha))\delta_1\\ (k_3(\alpha)
             +\beta k_4(\alpha))\delta_1\end{bmatrix}\right|\right|_2
           =\left|\left| \begin{bmatrix}  \overline{\alpha_1}\delta_1\\ \overline{\alpha_2}\delta_1 \end{bmatrix}\right|\right|_2\,,
\end{equation}
which implies
\begin{equation}\label{6.7-10-11}
1+|\beta|^2= |\alpha_1|^2+|\alpha_2|^2.
\end{equation}
Recall that
$\varphi_-=b_{\alpha}\theta_0^{\prime}\overline{a}$,
$\psi_-=b_{\alpha}\theta_1^{\prime}\overline{b}$,
$\varphi_+=b_\alpha \theta_1' \theta_3\overline{d}$, and
$\psi_+=b_\alpha \theta_0'\theta_2\overline{c}$. \
Thus from (\ref{6.7-77}), we can see that
\begin{equation}\label{6.7-11-11}
k_1=\theta_2k_1^{\prime}, \
k_2=\theta_1^{\prime}\theta_3k_2^{\prime}, \ k_3=\theta_0^{\prime}
\theta_2k_3^{\prime}, \ k_4=\theta_3k_4^{\prime},
\end{equation}
where $k_i^{\prime} \in H^{\infty}$ for $i=1,\cdots,4$. \
We claim that
\begin{equation}\label{6.7.11-12}
\hbox{$\theta_2$ and $\theta_3$ are both constant:}
\end{equation}
indeed, by (\ref{6.7-10-10}) and (\ref{6.7-11-11}),
$$
\aligned \alpha_1k_1+\alpha_2k_3=1 &\Longrightarrow
\alpha_1\theta_2k_1^{\prime}+\alpha_2\theta_0^{\prime}
\theta_2k_3^{\prime}=1\\
&\Longrightarrow
\theta_2(\alpha_1k_1^{\prime}+\alpha_2\theta_0^{\prime}
k_3^{\prime})=1\\
&\Longrightarrow
\overline{\theta_2}=\alpha_1k_1^{\prime}+\alpha_2\theta_0^{\prime}
k_3^{\prime}\in H^{\infty} \cap \overline{H^{\infty}}=\mathbb C\\
&\Longrightarrow\hbox{$\theta_2$ is constant}
\endaligned
$$
and
$$
\aligned \alpha_1k_2+\alpha_2k_4=\overline{\beta} &\Longrightarrow
\alpha_1\theta_1^{\prime}\theta_3k_2^{\prime}+\alpha_2\theta_3
       k_4^{\prime}=\overline{\beta}\\
&\Longrightarrow
\theta_3(\alpha_1\theta_1^{\prime}k_2^{\prime}+\alpha_2
        k_4^{\prime})=\overline{\beta}\\
&\Longrightarrow
\overline{\theta_3}=\frac{1}{\overline{\beta}}(\alpha_1\theta_1^{\prime}k_2^{\prime}+\alpha_2
k_4^{\prime})\in H^{\infty} \cap \overline{H^{\infty}}=\mathbb C\quad\hbox{(since $\beta\ne 0$)}\\
&\Longrightarrow\hbox{$\theta_3$ is constant\,,}
\endaligned
$$
which proves (\ref{6.7.11-12}). \
Without loss of generality, we may assume that
\begin{equation}\label{6.7-11-7}
\theta_2=\theta_3=1.
\end{equation}
We next claim that
\begin{equation}\label{6.7-11-4}
\theta_0=\theta_1=b_\alpha,\ \ \hbox{i.e.,}\ \
\hbox{$\theta_0^\prime$ and $\theta_1^\prime$ are both constant.}
\end{equation}
If $\varphi$ and $\psi$ are rational functions having the
same number of poles (this hypothesis has not been used until now) then we can see that
\begin{equation}\label{6.7-11-2}
\hbox{if $\theta_0^\prime$ or $\theta_1^\prime$ is constant then
         both $\theta_0^\prime$ and $\theta_1^\prime$ are constant:}
\end{equation}
indeed, since $\varphi_-=\theta_0\overline a$ and $\psi_-=\theta_1\overline b$ are rational functions,
it follows that $\theta_0$ and $\theta_1$ are finite Blaschke products, and hence by assumption,
$\hbox{deg}\, (\theta_0)
= \sharp \,\hbox{(poles of $\overline{\varphi_-}$)}
= \sharp \,\hbox{(poles of $\overline{\psi_-}$)}
=\hbox{deg}\, (\theta_1)$, giving
(\ref{6.7-11-2}).

\smallskip

Toward (\ref{6.7-11-4}), and in view of (\ref{6.7-11-2}), we assume to
the contrary that both $\theta_0'$ and $\theta_1'$ are not constant. \
Since $\theta_0'$ and $\theta_1'$ are non-constant finite Blaschke products,
there exist $v,w\in\mathbb D$ such that
$\theta_0'(v)=0=\theta_1'(w)$. \
But since $k_3=\theta_0'k_3'$ and
$k_2=\theta_1'k_2'$, it follows from (\ref{6.7-10-10}) that
\begin{equation}\label{6.7-11-5}
k_1(v)=\frac{1}{\alpha_1}\quad\hbox{and}\quad
k_4(w)=\frac{\overline\beta}{\alpha_2}
\end{equation}
(where we note that $\alpha_1\ne 0$ and $\alpha_2\ne 0$). \
Observe that $|k_1(v)|=1=|k_4(w)|$:
indeed, if $|k_1(v)|<1$, then $|\alpha_1|>1$, so that by (\ref{6.7-10-11}),
$|\alpha_2|<|\beta|$, which implies $|k_4(w)|>1$,
which contradicts the fact $||K||_\infty\le 1$ and if instead
$|k_4(w)|<1$, then similarly we get a contradiction. \
Since $||k_1||_\infty\le 1$ and
$||k_4||_\infty\le 1$, it follows from the Maximum Modulus Theorem
that $k_1$ and $k_4$ are both constant, i.e., \
\begin{equation}\label{6.7-11-6}
k_1=\frac{1}{\alpha_1}\quad\hbox{and}\quad
k_4=\frac{\overline\beta}{\alpha_2}.
\end{equation}
Then from (\ref{6.7-10-10}), we should have $k_2=k_3\equiv 0$, which leads to a contradiction, using
(\ref{6.7-77}). \


\bigskip

In view of (\ref{6.8-90}) and (\ref{6.7-11-4}), we can now write
$$
\Phi_-=\begin{bmatrix} b_\alpha & b_\alpha \overline b\\ b_\alpha \overline a&
b_\alpha\end{bmatrix}=B^*\Omega_2
      \quad\hbox{(left coprime factorization),}
$$
where
$$
\Omega_2:=\nu\begin{bmatrix} b_\alpha &-\overline{\gamma}b_\alpha\\
\gamma&1\end{bmatrix}.
$$
Also, in view of (\ref{6.7-11-7}) and (\ref{6.7-11-4}), we can write
$$
\begin{aligned}
\Phi_+ & = \begin{bmatrix} 0& b_\alpha\overline d\\ b_\alpha \overline
c&0\end{bmatrix}
   =\begin{bmatrix} 0&c\\ d&0\end{bmatrix}^*\begin{bmatrix} b_\alpha&0\\ 0& b_\alpha\end{bmatrix}\\
    &= A^*\Omega_0\Omega_2 =A^* \left( \nu \begin{bmatrix}1&\overline{\gamma}b_\alpha\\ -{\gamma}&b_\alpha\end{bmatrix}\right)
              \left(\nu \begin{bmatrix} b_\alpha&-\overline{\gamma}b_\alpha\\ \gamma &1\end{bmatrix}\right)\,,
\end{aligned}
$$
where $\Omega_0:=\nu \begin{bmatrix} 1&\overline{\gamma}b_\alpha\\
-{\gamma}&b_\alpha\end{bmatrix}$ and $\Omega_0\Omega_2= b_\alpha I_2$. \
Then by
(\ref{3.9}),
\begin{equation}\label{6.7-17}
\Omega_0 H^2_{\mathbb C^2} \subseteq \hbox{ker}\, [T_\Phi^*,
T_\Phi]\,,\quad\hbox{so that}\quad
\hbox{ran}\,[T_\Phi^*, T_\Phi]\subseteq
\mathcal{H}_{\Omega_0}\,.
\end{equation}
Since $\hbox{dim}\, \mathcal{H}_{\Omega_0}=1$, it follows
that $\hbox{ran}\,[T_\Phi^*, T_\Phi]=\mathcal{H}_{\Omega_0}$ or
$\hbox{ran}\,[T_\Phi^*, T_\Phi]=\{0\}$, i.e., $T_\Phi$ is normal. \
If
$T_\Phi$ is normal then the same argument as (\ref{6.18}) shows that
$$
\varphi=\mu\, \overline b_\alpha + e^{i
\theta}\sqrt{1+|\mu|^2}\,b_\alpha+\zeta\quad\hbox{and}\quad
\psi=e^{i\,(\pi-2\,{\rm arg}\,\mu)}\varphi,
$$
where $|\mu|=1,\ \zeta\in\mathbb C,$ and $\theta\in
[0,2\pi)$.

\medskip

Suppose $\hbox{ran}\,[T_\Phi^*, T_\Phi]=\mathcal{H}_{\Omega_0}$. \
We
now recall a well-known result of B. Morrel (\cite{Mor};
\cite[p.162]{Con}). If $T\in\mathcal{B(H)}$ satisfies the following
properties: (i) $T$ is hyponormal; (ii) $[T^*,T]$ is rank-one;
and (iii) $\hbox{ker}\,[T^*,T]$ is invariant for $T$, then $T-\beta$
is quasinormal for some $\beta\in\mathbb C$, i.e., $T-\beta$
commutes with $(T-\beta)^*(T-\beta)$. \
Since $T_\Phi$ satisfies the
above three properties, we can conclude that $T_{\Phi-\beta}$ is
quasinormal for some $\beta\in\mathbb C$. \
Thus,
$T_{\Phi-\beta}^*[T_{\Phi-\beta}^*,\, T_{\Phi-\beta}]=0$. \
But since
$[T_{\Phi-\beta}^*, T_{\Phi-\beta}]=[T_{\Phi}^*, T_{\Phi}]$, it
follows that
\begin{equation}\label{6.7-17-17}
T_{(\Phi^*-\overline{\beta})}[T_{\Phi}^*, T_{\Phi}]=0.
\end{equation}
Observe that
$$
\Omega_0=\nu \begin{bmatrix}1&\overline\gamma b_\alpha \\ -\gamma
&b_\alpha\end{bmatrix}=
  \nu \begin{bmatrix}1&\overline \gamma\\ -\gamma &1\end{bmatrix} \cdot
     \begin{bmatrix}1&0\\ 0&b_\alpha\end{bmatrix} \,,
$$
so that
$$
\begin{bmatrix} \overline\gamma\delta\\ \delta \end{bmatrix}= \begin{bmatrix}1&\overline \gamma\\ -\gamma &1\end{bmatrix}
\,\begin{bmatrix} 0\\ \delta \end{bmatrix} \in \mathcal H_{\Omega_0}
\quad\hbox{($\delta:=\frac{\sqrt{1-|\alpha|^2}}{1-\overline\alpha z}$)}\,,
$$
which implies that $\begin{bmatrix} \overline\gamma\delta \\ \delta\end{bmatrix}\in
\hbox{ran}\,[T_\Phi^*, T_\Phi]$.

On the other hand, since $a,b,c,d\in\mathcal{H}_{zb_\alpha}$,
if we choose $\{1, b_\alpha\}$ as a (not necessarily orthogonal) basis for
$\mathcal{H}_{zb_\alpha}$, then $\overline a {b_\alpha}$ and $\overline b {b_\alpha}$ are
of the form $\xi_1 {b_\alpha}+\xi_2$ ($\xi_1,\xi_2\in\mathbb C$) and
$c {\overline b_\alpha}$ and $d {\overline b_\alpha}$ are of the form
$\eta_1 {\overline b_\alpha}+\eta_2$ ($\eta_1,\eta_2\in\mathbb C$). \
Thus we may assume that $T_\Phi^*$ is of the form
$$
T_\Phi^*=
\begin{bmatrix} T_{b_\alpha} & \overline b T_{b_\alpha}+  cT_{\overline b_\alpha}+c_0\\
\overline a T_{b_\alpha}+ d T_{\overline b_\alpha}+d_0 & T_{b_\alpha}\end{bmatrix}\quad
(a,b,c,d,c_0,d_0\in\mathbb C)\,.
$$
Then by (\ref{6.7-17-17}) we have
$$
\begin{bmatrix} T_{b_\alpha} & \overline b T_{b_\alpha}+  cT_{\overline b_\alpha}+c_0\\
\overline a T_{b_\alpha}+ d T_{\overline b_\alpha}+d_0 & T_{b_\alpha}\end{bmatrix}
\begin{bmatrix} \overline\gamma\delta \\ \delta\end{bmatrix}
   =\begin{bmatrix}\overline\gamma \overline\beta\delta\\
           \overline{\beta}\delta \end{bmatrix}\,.
$$
Now recall the case assumption, which gives $\gamma=-b=-\frac{1}{a}$; it
follows that
\begin{equation}\label{6.7.9}
\begin{bmatrix}\overline\gamma \overline\beta \\ \overline{\beta} \end{bmatrix}
  = \begin{bmatrix} (\overline\gamma +\overline b) b_\alpha +
             c_0\\ (1+\overline a \overline\gamma) b_\alpha +\overline\gamma d_0\end{bmatrix}
        = \begin{bmatrix} c_0\\ \overline\gamma d_0 \end{bmatrix}.
\end{equation}
which implies that
$$
d_0=\frac{\overline\beta}{\overline\gamma}=\frac{1}{\overline\gamma^2}
c_0 =\overline a^2 c_0.
$$
On the other hand, a straightforward calculation shows that
$$
[T_\Phi^*, T_\Phi]=\begin{bmatrix} A&*\\ *&*\end{bmatrix},
$$
where
\begin{equation}\label{6.7.6}
\begin{aligned}
A & := \Bigl((|c|^2+|c_0|^2) -(1+|a|^2+|d_0|^2)\Bigr)
        + \Bigl(\overline{bc_0}+c_0\overline c - \overline d d_0 -\overline{d_0}\overline a\Bigr)T_{b_\alpha}\\
             & +\Bigl(bc_0 +c\overline{c_0}- d\overline{d_0}-d_0 a \Bigr) T_{\overline b_\alpha}
                 + \Bigl( \overline{bc} - \overline{ad} \Bigr) T_{b_\alpha^2}
                 + \Bigl( bc-ad \Bigr) T_{\overline b_\alpha^2}
                    + (1+|b|^2-|d|^2) T_{b_\alpha} T_{\overline b_\alpha}.
\end{aligned}
\end{equation}
But since $\hbox{rank}\, A\le 1$, we have
\begin{equation}\label{6.7.5}
\begin{cases}
     bc=ad\\
      bc_0 +c\overline{c_0}- d\overline{d_0}-d_0 a=0.
\end{cases}
\end{equation}
Since $ab=1$ (and hence, $c=a^2d$) and $\overline{d_0}=a^2
\overline{c_0}$, we have $c\overline{c_0}-
d\overline{d_0}=a^2d\overline{c_0}-d a^2 \overline{c_0}=0$. \
Thus by
(\ref{6.7.5}),
\begin{equation}\label{6.7.8}
bc_0=ad_0=a\overline a^2c_0.
\end{equation}
If $c_0\ne 0$, then by (\ref{6.7.8}), $\frac{1}{a}=a\overline a^2$,
i.e., $|a|=1$, and in turn $|c_0|=|d_0|$. \
Also since $ab=1$, we can
write
$$
a=e^{i\theta}\quad\hbox{and}\quad b=e^{-i\theta}\quad\hbox{for some
$\theta\in [0,2\pi)$}.
$$
Since $|c_0|=|d_0|$ and by (\ref{6.7.6}),
$$
|c|^2+|c_0|^2-(1+|a|^2+|d_0|^2)+(1+|b|^2-|d|^2)=0,
$$
it follows that $|c|=|d|$. \
Moreover, a straightforward calculation
shows that
$$
[T_\Phi^*, T_\Phi]=\begin{bmatrix} |d|^2-2 & -2 e^{i\theta}\\ -2
e^{-i\theta}& |d|^2-2
\end{bmatrix}\, K_0\quad (\hbox{where $K_0:=1-T_{b_\alpha}T_{\overline b_\alpha}$}).
$$
Since $\hbox{rank}\,[T_\Phi^*, T_\Phi]=1$, it follows that
$(|d|^2-2)^2-4=0$, i.e., $|d|=2$. \
Thus we can write
$$
a=e^{i\theta},\ \ b=e^{-2i\theta} a,\ \ \overline d=2 e^{i\omega},\
\ \overline c=e^{-2i\theta} \overline d\quad\hbox{(since
$bc=ad$)}\,.
$$
Also since $\overline{c_0}=\overline
a^2\overline{d_0}=e^{-2i\theta}\overline{d_0}$, it follows that
\begin{equation}\label{6.7.7}
\begin{cases}
\varphi=a\overline b_\alpha +\overline d b_\alpha +\overline{d_0}= e^{i\theta}\overline b_\alpha +2 e^{i\omega} b_\alpha +\overline{d_0}\\
\psi=b\overline b_\alpha + \overline c b_\alpha +\overline{c_0}=
e^{-2i\theta}\varphi.
\end{cases}
\end{equation}
In particular, since
$\beta=\gamma\overline{d_0}=-\frac{1}{a}\overline{d_0}=-e^{-i\theta}\overline{d_0}$,
it follows that $T_\Phi-\beta=T_\Phi+e^{-i\theta}\overline{d_0}$ is
quasinormal. \

If $c_0=0$ then by (\ref{6.7.9}), $\beta=0$, and hence $d_0=0$. In
particular, $T_\Phi$ is quasinormal. \
A straightforward calculation
shows that
$$
[T_\Phi^*, T_\Phi]=\begin{bmatrix}
A & -(a+\overline b)K_0\\
-(b+\overline a)K_0 & B
\end{bmatrix}\,,
$$
where
$$
\begin{cases}
A:= |c|^2-1-|a|^2+(bc-ad) T_{\overline b_\alpha^2}+\overline{(bc-ad)}\, T_{b_\alpha^2}+(1+|b|^2-|d|^2) T_{b_\alpha}T_{\overline b_\alpha} \\
B:= |d|^2-1-|b|^2+(ad-bc) T_{\overline b_\alpha^2}+\overline{(ad-bc)}\, T_{b_\alpha^2}+(1+|a|^2-|c|^2) T_{b_\alpha}T_{\overline b_\alpha}\\
K_0:=1-T_{b_\alpha}T_{\overline b_\alpha}.
\end{cases}
$$
Since $\hbox{rank}\,[T_\Phi^*, T_\Phi]=1$, we have
$$
\begin{cases}
bc=ad\\
|c|^2-1-|a|^2=|d|^2-|b|^2-1.
\end{cases}
$$
We thus have
$$
[T_\Phi^*, T_\Phi]=\begin{bmatrix}
|c|^2-1-|a|^2 & -(a+\overline b)\\
-(b+\overline a) & |c|^2-1-|a|^2
\end{bmatrix}\, K_0\quad (\hbox{where}\ K_0:=1-T_{b_\alpha}T_{\overline b_\alpha}).
$$
Thus
$$
\begin{aligned}
0=T_\Phi^*[T_\Phi^*, T_\Phi]
&= \begin{bmatrix} T_{b_\alpha} & \overline b T_{b_\alpha}+  cT_{\overline b_\alpha}\\
     \overline a T_{b_\alpha}+ d T_{\overline b_\alpha}& T_{b_\alpha}\end{bmatrix}
        \begin{bmatrix} |c|^2-1-|a|^2 & -(a+\overline b)\\ -(b+\overline a) & |c|^2-1-|a|^2 \end{bmatrix}\, K_0\\
&= \begin{bmatrix} \Bigl((|c|^2-1-|a|^2)-\overline b(b+\overline a)\Bigr) T_{b_\alpha}K_0 &\ast\\
               \ast & \Bigl(-\overline a(a+\overline b)+ (|c|^2-1-|a|^2)\Bigr) T_{b_\alpha}K_0
     \end{bmatrix}\,,
\end{aligned}
$$
which implies
$$
\begin{cases}
|c|^2-1-|a|^2-|b|^2-1=0\\
|c|^2-1-|a|^2-|a|^2-1=0,
\end{cases}
$$
giving $|a|=|b|=1$ and in turn $|c|=|d|=2$. \
As in  (\ref{6.7.7}), we may thus write
$$
\begin{cases}
\varphi= e^{i\theta}\overline b_\alpha +2 e^{i\omega} b_\alpha\\
\psi=e^{-2i\theta}\varphi.
\end{cases}
$$

This completes the proof.
\end{proof}

\bigskip

We can say more about the solution of the case (\ref{6.6-0}).

\medskip

\begin{cor}
Using the terminology in case (\ref{6.6-0}), assume that
either $\varphi$ or $\psi$ is a rational function having at least two poles. \
Then both of $\varphi$ and $\psi$ are rational. \
Moreover, in this case, either $\varphi$ or $\psi$ has exactly one pole, say $\alpha$.
\end{cor}

\begin{proof}
Suppose either $\varphi$ or $\psi$ is a rational function having at least two poles. \
Thus either $\theta_0'$ or $\theta_1'$ is a nonconstant finite Blaschke product. \
Without loss of generality we assume that  $\theta_0'$ is a nonconstant finite Blaschke product. \
If $\theta_1'$ has a nonconstant Blaschke factor then the same argument as in (\ref{6.7-11-5})
leads to a contradiction. \
Therefore, for the first assertion, we assume to the contrary that
$\theta_1'$ is a nonconstant singular inner function. \
Since $\theta_0'$ is a nonconstant finite Blaschke product,
$$
\exists \ w\in\mathbb D\ \hbox{such that}\ \theta_0'(w)=0,\ \hbox{so that by (\ref{6.7-11-11})},\
k_3(w)=0.
$$
Thus by (\ref{6.7-10-10}), $k_1(w)=\frac{1}{\alpha_1}$. \
But since $|k_1(w)|<1$ (if it were not so, then $k_1$ would be constant,
so that $k_3\equiv 0$, a contradiction from (\ref{6.7-77})),
it follows that $1<|\alpha_1|$. \
Thus by (\ref{6.7-10-11}),
\begin{equation}\label{6.8-0}
|\alpha_2|<|\beta|.
\end{equation}
On the other hand, since $\theta_1'$ is a nonconstant singular inner function,
we can see that there exists
$\delta\in [0,2\pi)$ such that $\theta_1'$ has nontangential limit $0$ at $e^{i\delta}$
(cf. \cite[Theorem II.6.2]{Ga}). \
Thus  by (\ref{6.7-11-11}), $k_2$ has nontangential limit $0$ at $e^{i\delta}$ and in
turn, by
 (\ref{6.7-10-10}),  $k_4$ has nontangential limit $\frac{\overline\beta}{\alpha_2}$ at $e^{i\delta}$. \
But since $||k_4||_\infty\le 1$, it follows that
$\left|\frac{\overline\beta}{\alpha_2}\right|\le 1$, i.e., $|\beta|\le |\alpha_2|$,
which contradicts (\ref{6.8-0}). \
This proves the first assertion. \
The second assertion follows at once from the same argument as in (\ref{6.7-11-5}):
in other words, either $\theta_0'$ or $\theta_1'$ is constant, i.e.,
$\theta_0=b_\alpha$ or $\theta_1=b_\alpha$. \
This complete the proof.
\end{proof}


\begin{rem} \ Due to a technical problem, we omitted a detailed
proof for the case B-2 from the proof of \cite[Theorem 5.1]{CHL1}. \
The proof of the case B-2 (with $\alpha=0$) in the proof of Theorem \ref{thm4.2}
provides the portion of the
proof that did not appear in \cite{CHL1}. \
In particular, Theorem
\ref{thm4.2} incorporates  an extension of a corrected version of \cite[Theorem 5.1]{CHL1},
in which the exceptional case (\ref{6.6-0}) was
omitted. \
\end{rem}

%
%


\end{document}